%% file: Weight.tex
\documentclass[twoside, 11pt]{amsart}

\usepackage{amsfonts}
\usepackage[utf8]{inputenc}
\usepackage{array}

\usepackage[english]{babel}
\usepackage{latexsym}
\usepackage{amscd}
\usepackage{amsmath}
\usepackage{amssymb}
\usepackage{amsthm}
\usepackage{graphicx}
\usepackage{mathrsfs}
\usepackage{mathtools}
\usepackage{pict2e}
\usepackage{stackrel} 
\usepackage{wasysym}
\usepackage[all]{xy}
\usepackage[dvipsnames]{xcolor}
\usepackage[colorlinks,final,hyperindex]{hyperref}
\usepackage[noabbrev,capitalize]{cleveref}
\usepackage{tikz}
\usepackage{tikz-cd}
\usetikzlibrary{decorations.pathmorphing,decorations.markings,arrows,calc,shapes.geometric,arrows.meta,positioning}

\usepackage{geometry}

\allowdisplaybreaks

\newtheorem{lemma}{Lemma}[section]

\newtheorem{theorem}[lemma]{Theorem}

\newtheorem{corollary}[lemma]{Corollary}
\newtheorem{proposition}[lemma]{Proposition}

\newtheorem*{theostar}{\sc Theorem}
\theoremstyle{definition}
\newtheorem{definition}[lemma]{Definition}
\newtheorem{remark}[lemma]{\sc Remark}
\newtheorem{example}[lemma]{\sc Example}

\newtheorem{assumption}{\sc Assumptions}

\newcommand{\K}{\mathrm{K\ddot{a}h}}
\newcommand{\nocontentsline}[3]{}
\newcommand{\tocless}[2]{\bgroup\let\addcontentsline=\nocontentsline#1{#2}\egroup}

\hypersetup{citecolor=BleuMinuit, linkcolor=bordeau, filecolor=black, urlcolor=BleuMinuit}

\author{Coline Emprin}

\author{Geoffroy Horel}

\title{Weight structures and formality}
\date{\today}

\address{Coline Emprin, D\'epartement de math\'ematiques et applications, \'Ecole normale sup\'erieure,
	\indent  45 rue d’Ulm, 75230 Paris, France. }
\email{\noindent \href{mailto:coline.emprin@ens.psl.eu}{coline.emprin@ens.psl.eu} \medskip}

\address{Geoffroy Horel, Universit\'e Sorbonne Paris Nord, Laboratoire de Géométrie, Analyse et \indent Applications, CNRS, UMR 7539, F-93430, Villetaneuse, France}
\email{\href{mailto:horel@math.univ-paris13.fr}{horel@math.univ-paris13.fr}}

\thanks{
\indent Both authors were supported ANR-20-CE40-0016 HighAGT}


\begin{document}

\include{macro}

\begin{abstract}
 This is a survey on formality results relying on weight structures. A weight structure is a naturally occurring grading on certain differential graded algebras. If this weight satisfies a purity property, one can deduce formality. Algebraic geometry provides us with such weight structures as the cohomology of algebraic varieties tends to present additional structures including a Hodge structure or a Galois action. 
\end{abstract}

\maketitle

\setcounter{tocdepth}{1}
\tableofcontents

\section*{\textcolor{bordeau}{Introduction}}

\noindent \textbf{The notion of formality.} Let $A$ be a chain complex equipped with an algebraic structure (e.g.\ an associative algebra, a commutative algebra, an operad, etc.). The homology of this complex inherits the same type of structure. In general, this induced structure does not retain all the homotopical information contained in $A$. For example, it is well-known that the homology of a differential graded algebra can have additional non-trivial \emph{Massey products} that witness homotopical information about the algebra. In some cases all these Massey products vanish, namely if $A$ is homotopy equivalent to an algebra whose differential is identically zero. If so, the algebra is said to be \emph{formal}. The idea of formality originated in rational homotopy theory. In this context, a topological space $X$ is formal if its Sullivan’s algebra of polynomial forms $\Omega^*_{PL}(X)$ is connected to its cohomology $H^*(X; \mathbb{Q})$ by a string of quasi-isomorphisms of commutative differential graded algebras. \bigskip

\noindent \textbf{The case of compact Kähler manifolds.} A central formality result was proved by Deligne, Griffiths, Morgan, and Sullivan in \cite{DGMS75}. Using Hodge theory, they showed that any compact Kähler manifold is formal, see Section \ref{2}. However, as explained in their introduction, their intuition came from the Weil's conjectures and the following observation~:  formality of an algebra can be viewed as a multiplicative splitting of the canonical filtration. Indeed, if that is the case, the algebra is quasi-isomorphic to the associated graded of the canonical filtration, which is exactly the cohomology algebra. The abstract statement that emerges from this intuition is the following one.

\begin{theostar}
If a dg-algebra $A^*$ admits a multiplicative ``weight decomposition'' 
\[A^*=\oplus_{i\in\mathbb{Z}} A^*_i\]
with the property that $H^i(A)$ is concentrated in weight $i$, then $A$ is formal. 
\end{theostar}

A functorial version of this theorem appears as Proposition \ref{factorisation}, whose proof is basic. The strength of this result comes from combining it with deep results from algebraic geometry producing the desired weight decompositions. \bigskip

\textbf{Mixed Hodge structures.}  The first type of such decompositions studied in this survey comes from mixed Hodge theory. By work of Deligne, the cohomology of complex algebraic varieties carries a canonical mixed Hodge structure. It turns out that this structure can be lifted at the chains level. Moreover, a result of Deligne shows that mixed Hodge structure are functorialy split. From these two facts, we can obtain a functorial weight decomposition as in the Theorem above for many varieties. We refer the reader to Section \ref{Hodge} for a survey of this result based on joint work of the second author and Joana Cirici, see \cite{CH20}. \bigskip

\textbf{Galois group actions.} One can try to use similar techniques in order to prove formality results over $\mathbb{F}_p$ instead of $\mathbb{Q}$. In that case, mixed Hodge theory does not make sense anymore. However, one can use Frobenius actions on \'etale cohomology. In the context a smooth projective variety over a finite field, the canonical filtration can be split by eigenvalues of a Frobenius action. This idea is explained in Section \ref{torsion} based on a joint paper of the second author with Joana Cirici, see \cite{CH22}. In this context, one does not obtain full $\mathbb{Z}$-graded weight decompositions as in the Theorem above but merely $\mathbb{Z}/(h)$-graded decompositions for some integer $h$. \bigskip

\textbf{Gauge formality.} By using the operadic calculus, one can have another approach to formality which boils down to a deformation problem. This leads to the notion of \emph{gauge formality} which is presented in Section \ref{7}. We present a joint work of Gabriel Drummond-Cole and the second author (see \cite{DCH21}) which revisit the results of the previous sections, using the approach of gauge formality. This idea of gauge formality gives rise to an obstruction theory to formality due initially to Kaledin \cite{Kal07} in the context of associative algebras and pushed further by Melani-Rubio and the first author, see \cite{MR19,Kaledin}. One significant consequence of this theory are formality descent results, see Theorem \ref{descent2}. \bigskip

\noindent \textbf{Notations and conventions}. 

\begin{itemize}
	\item[$\centerdot$] Let $R$ be a commutative ground ring.
	\item[$\centerdot$] We generically write $\otimes$ for the tensor product over a commutative ground ring that should always be clear from context.
	\item[$\centerdot$] If $A$ is a chain complex and $x \in A$ is a homogeneous element, we denote by $|x|$ its homological degree.
	\item[$\centerdot$] The abbreviation ``dg'' stands for the words ``differential graded''.
	\item[$\centerdot$]  We use the notations of \cite{LodayVallette12} for operads.  
\end{itemize}

\noindent \textbf{Acknowledgments}. This survey emerged from a series of three lectures given by the second named author at the workshop Higher Structures and Operadic Calculus at CRM Barcelona in June 2021. We would like to thank all participants and organizers. We are deeply grateful to Joana Cirici and Clément Dupont for a thorough review of these notes.  \bigskip

\section{\textcolor{bordeau}{The notion of formality}}\label{2.}

In this section, we define the formality of various algebraic structures and discuss its origins in rational homology theory.

\subsection{Formality of algebraic structures}

Let $\P$ be an operad in $R$-modules. Let $(A, \phi)$ be a dg $\P$-algebra, i.e. a chain complex $A$ over $R$ endowed with an operad morphism \[\varphi : \P \longrightarrow \End_A \] between $\P$ and the endomorphism operad associated to $A$. The induced map $\varphi_* : \P \to \End_{H(A)}$ turns the homology into a $\P$-algebra. We will refer to it as the \emph{induced structure} in homology. 

\begin{definition}[Formality of dg $\P$-algebras]\label{formal}
The dg $\P$-algebra $(A, \varphi)$ is \emph{formal} if there exists a zig-zag of dg $\P$-algebra quasi-isomorphisms
\[ (A, \varphi) \; \overset{\sim}{\longleftarrow} \;  \cdot \;  \overset{\sim}{\longrightarrow} \;  \cdots \;  \overset{\sim}{\longleftarrow} \; \cdot \;  \overset{\sim}{\longrightarrow} \; (H(A) , \varphi_*) \] relating it to its induced structure in homology.
\end{definition}

\begin{remark}
The number of quasi-isomorphisms involved in a formality zig-zag is arbitrary. However, it can be reduced to a length two zig-zag in many cases, e.g. if $R$ is a field, under additional assumptions on $\P$ in the positive characteristic case. Under these assumptions, the category of dg $\P$-algebras is equipped with a transferred model category structure, see \cite[Theorem 4.1.1]{Hinich97}. Then, any zig-zag of quasi-isomorphisms induces an isomorphism in the associated homotopy category and can be represented by an actual weak equivalence between a cofibrant replacement of the source and a fibrant one of the target. Any object being fibrant in this context, this leads to a zig-zag
\[(A, \varphi) \; \overset{\sim}{\longleftarrow} \; Q(A, \varphi) \; \overset{\sim}{\longrightarrow} \; (H(A) , \varphi_*) \ ,\] where $Q(A, \varphi)$ denotes a cofibrant replacement of $(A, \varphi)$.
\end{remark}

\begin{definition}[Lax symmetric monoidal functor]\label{lax monoidal functor}
A \emph{lax monoidal functor} \[(F, \kappa, \eta) : (\mathcal{C}, \otimes, \mathbf{1}) \longrightarrow (\mathcal{D}, \otimes, \mathbf{1}')\] is a functor $F : \mathcal{C} \to \mathcal{D}$ between monoidal categories together with maps
\[\kappa_{X,Y} : F(X) \otimes F(Y) \longrightarrow F(X \otimes Y)\]
that are natural in the objects $X$ and $Y$ of $\mathcal{C}$, and a morphism of $\mathcal{D}$, \[\eta : \mathbf{1}' \longrightarrow F(\mathbf{1}) \] that are compatible with the constraints of associativity and unit. The functor $F$ is said to be \emph{lax symmetric monoidal}, if $\kappa$ is compatible with the commutativity constraint. A lax monoidal functor is called \emph{strong} if $\kappa$ and $\eta$ are isomorphisms.  We refer the reader to \cite[Chapter 2]{etingoftensor} for more details.
\end{definition}

\begin{remark}
In the sequel, strong monoidal functors will not play an important role and we shall often use ``symmetric monoidal functor'' to refer to a lax symmetric monoidal functor.
\end{remark}

\begin{example}
Let $(\mathcal{A}, \otimes, \mathbf{1})$ be an abelian symmetric monoidal category with infinite direct sums. The homology functor $H : \ch_*(\mathcal{A}) \to \ch_*(\mathcal{A})$ is lax symmetric monoidal, via the usual Künneth morphism. If $\mathcal{A}$ is the category of vector spaces over a field, then this functor is strong symmetric monoidal.
\end{example}

\begin{definition}[Formality of symmetric monoidal functors]\label{formal2}
	Let $(\mathcal{C}, \otimes, \mathbf{1})$ a symmetric monoidal category. A symmetric monoidal functor $F : \mathcal{C} \to \ch_*(R)$ is \emph{formal} if it is weakly equivalent to $H \circ F$, i.e. if there exists a zig-zag of natural transformations of symmetric monoidal functors  $$ F \; \overset{\Phi_1}{\longleftarrow} \; F_1 \; \longrightarrow \; \cdots \; \longleftarrow \;  F_n \;  \overset{\Phi_n}{\longrightarrow} \; H \circ F$$ such that $\Phi_i(X)$ is a quasi-isomorphisms for every object $X$ of $\mathcal{C}$. 
\end{definition}

If we allow operads to be \emph{colored}, the formality of symmetric monoidal functors appears as a particular case of Definition \ref{formal} (see Proposition \ref{equivalence} below). A colored operad is an operad in which each input or output comes with a color chosen in a given set. A composition is possible whenever the colors of the corresponding input and output involved match. We now give a precise definition. Let $\langle n\rangle$ denote the finite set $\{0,1,\ldots,n\}$.

\begin{definition}[Set colored operads]
 Let $I$ be a set of colors. Fix $(\mathcal{C},\otimes,\mathbf{1})$ a symmetric monoidal category. An \emph{$I$-colored operad} in $\mathcal{C}$ is a set \[\lbrace\P(n,i)\rbrace_{i : \langle n\rangle \to I} \] of objects of $\mathcal{C}$ indexed by all maps $i : \langle n\rangle \to i$, for $n \geqslant  0$, together with 
\begin{itemize}
	
	\item[$\centerdot$] composition maps for all $l \leqslant  n$, $$\circ_{l} : \P(n,i) \otimes \P(m,j) \to \P(m+n -1, i \circ_{l} j) $$ where  $i(l) = j(0)$ and  $i \circ_{l} j : \langle m +n -1\rangle  \to I$ is defined by
	
	\[i \circ_{l} j(k)  =\left\{ \begin{array}{lll}
		i(k)  &  \mbox{if }   k < l \\
		j(k-l+1)
		  &  \mbox{if } l  \leqslant  k < l + m \\
		i(k - m)  &  \mbox{if }  l+m \leqslant  k \ ;\ \\
	\end{array} \right.
	\] \item[$\centerdot$] a right $\mathbb{S}_n$-action on \[\bigoplus_{i: \langle n\rangle\to I} \P(n,i) \ ; \]	
	\item[$\centerdot$] an identity $\mathrm{id}_{\alpha} \in \P(1, c_{\alpha})$, for each $\alpha \in I$, where $c_{\alpha} : \langle 1\rangle  \to I$ is the constant map with value $\alpha$. These identities act as units with respect to any well defined composition.
\end{itemize}
These data satisfy the compatibility relations for $\circ_{l}$-operations of an operad (associativity, equivariance, etc.) whenever these make sense. 
\end{definition}

\begin{example}[The endomorphism colored operad]
	 Let $A = \lbrace A_{\alpha}  \rbrace_{\alpha \in I}$ be a family of chain complexes. The associated endomorphism $I$-colored operad $\End_A$ is defined for all $n \geqslant  0$, and all $i : \langle n\rangle \to I$, by $$\End_A(n,i)= \Hom\left(A_{i(1)} \otimes \cdots \otimes A_{i(n)}, A_{i(0)} \right) \ , $$ where the composition products (resp. the $\mathbb{S}_n$-actions) are induced by substitution (resp. permutation) of the tensor factors.  
\end{example}

\begin{definition}[Algebras over a set colored operad]
Let $\P$ be a $I$-colored operad. A dg $\P$-algebra is a family $A = \lbrace A_{\alpha}  \rbrace_{\alpha \in I}$ of chain complexes endowed with a morphism of $I$-colored operads $$ \P \longrightarrow \End_A \ .$$  
\end{definition}

\begin{remark}
In order to encode symmetries, one can also consider groupoid colored operads where colors are chosen in a given groupoid $\mathbb{V}$ instead of a set $I$. The associated Koszul duality theory was developed by Ward in \cite{War19}. We refer the reader to \cite[Section 5]{RL2} for more details. 
\end{remark}

\begin{example}
There exists an $\mathbb{N}$-colored operad $\mathcal{O}$ such that $\mathcal{O}$-algebras are exactly non-symmetric operads, see \cite[Section~4]{vdL03}. Similarly, there exists an $\mathbb{N}$-colored operad encoding symmetric operads (see \cite[Definition 5.1.5]{chuhaugseng})
\end{example}

\begin{proposition}\label{equivalence}
Let $(\mathcal{C}, \otimes, \mathbf{1})$ be a symmetric monoidal category. There is an associated $Ob(\mathcal{C})-$colored operad defined for all $i : \langle n\rangle \to Ob(\mathcal{C})$ by
\[\mathcal{Q}(n,i) \coloneqq \Hom_{\mathcal{C}}(i(1)\otimes \cdots \otimes i(n), i(0) ) \ .\]
A dg $\mathcal{Q}$-algebra over this operad is the same data as a lax symmetric monoidal functor 
\[F : \mathcal{C} \to \ch_*(R) \ .\]
\end{proposition}

\begin{proof}
Let $\lbrace A_{\alpha} \rbrace_{\alpha \in Ob(\mathcal{C})}$ be a family of chain complexes and let 
\[\varphi : \mathcal{Q} \to \End_A \] 
be a dg $\mathcal{Q}$-algebra structure. Setting $F(\alpha) \coloneqq A_{\alpha}$ for all $\alpha \in Ob(\mathcal{C})$, we obtain a symmetric monoidal functor $\mathcal{C} \to \ch_*(R)$ such that
\[\kappa_{\alpha,\beta} \coloneqq \varphi(2,i)(\mathrm{id}_{\alpha \otimes \beta}) \]
where $i : \langle 2\rangle \to Ob(\mathcal{C})$ is defined by $i(0) = \alpha \otimes \beta$, $i(1) = \alpha$ and $i(2) = \beta$. Conversely, out of a symmetric monoidal functor $F : \mathcal{C} \to \ch_*(R)$ one defines a dg $\mathcal{Q}$-algebra $(A,\varphi)$ with 
\[A \coloneqq \lbrace F(c) \rbrace_{c \in Ob(\mathcal{C})}  \quad \mbox{and}  \quad \varphi(n,i) \coloneqq F \circ \kappa\] for all $i : \langle n\rangle \to Ob(\mathcal{C})$, where $\kappa$ denotes the successive compositions giving \[F(i(1)) \otimes \cdots \otimes F(i(n)) \to F(i(1) \otimes \cdots i(n)) \ . \qedhere \]  
\end{proof}

The two following propositions are direct consequences of the definitions. 

\begin{proposition}\label{ima}
Let $\P$ be an operad in sets. If $F : \mathcal{C} \to \ch_*(R)$ is a formal symmetric monoidal functor and if $A$ is a $\P$-algebra in $\mathcal{C}$ (resp. operad), then $F(A)$ is a formal $\P$-algebra (resp. operad).
\end{proposition}

\begin{proof}
A lax monoidal functor sends $\P$-algebras to $\P$-algebras. A natural transformation between lax monoidal functors sends $\P$-algebras to morphisms of $\P$-algebras. If we evaluate the zig-zag connecting $F$ to $H\circ F$ on $A$, we obtain a formality zig-zag for $A$.
\end{proof}

\begin{proposition}\label{composition} 
 	Let $U : \mathcal{B} \to \ch_*(R)$ be a formal symmetric monoidal functor. For every symmetric monoidal functor $F : \mathcal{C}\to \mathcal{B}$, the composition
 	\[U \circ F : \mathcal{C}\to \ch_*(R) \]
is a formal symmetric monoidal functor.  
\end{proposition}

\begin{remark}[An application of the formality of an operad]
	Over a characteristic zero field $\mathbf{k}$, there is a Quillen equivalence between algebras over a formal dg operad $\P$ and algebras encoded by its homology $H(\P)$, see \cite[Theorem 4.7.4]{Hinich97}. This leads to Quillen equivalences between
	\begin{itemize}
		\item[$\centerdot$] $A_\infty$-algebras and associative algebras; 
		\item[$\centerdot$] $C_\infty$-algebras and commutative algebras;
		\item[$\centerdot$] $L_\infty$-algebras and Lie algebras;
		\item[$\centerdot$] Gerstenhaber algebras and algebras over the operad $C_*(\mathcal{D}_2;\mathbf{k})$, since the little disks operad $\mathcal{D}_2$ is formal and its homology is the Gerstenhaber operad, see Example \ref{little disks operad}. 
	\end{itemize}
	Kontsevich formality theorem can be improved using the last Quillen equivalence. Given a smooth manifold $M$, it asserts that the Gerstenhaber algebra of polyvector fields $\Gamma(\Lambda TM)$ is weakly equivalent to the $C_*(\mathcal{D}_2)$-algebra of Hochschild cochains on the algebras of smooth functions on $M$. This version of Kontsevich formality is due to Tamarkin and generalizes the classical formulation in terms of $L_\infty$-algebras. The fact that the Hochschild cochain complex carries an action of $C_*(\mathcal{D}_2)$ is a highly non-trivial theorem called Deligne's conjecture and was initially proved by Tamarkin, see \cite{Tam98,hinich03bis}.
\end{remark}

\subsection{Origins in rational homotopy theory}\label{Key}

The idea of formality originated in the field of rational homotopy theory. For an overview of rational homotopy theory, we refer the reader to \cite{BS24}. Very briefly, the rational homotopy category \[\mathrm{Ho(Top)}_{\mathbb{Q}}\] is obtained from the category of topological spaces by inverting maps that induce isomorphisms on homology with rational coefficients. This is a localization of the usual homotopy category in which one only inverts weak homotopy equivalences. The set of morphisms in this category are much more computable than in the usual homotopy category and still capture interesting invariants of homotopy types. For instance, the rational homotopy groups of a simply connected topological space $X$ can be computed as maps from a sphere to $X$ in the rational homotopy category:
\[\pi_n(X)\otimes_{\mathbb{Z}}\mathbb{Q}\cong [S^n,X]_{\mathrm{Ho(Top)}_{\mathbb{Q}}}.\]

The most fundamental theorem in the field of rational homotopy theory is due to Sullivan. It relies on the construction of a functorial commutative differential graded algebra (CDGA), the Sullivan algebra of polynomial forms 
\[X\mapsto \Omega^*_{PL}(X) \ ,\] 
that faithfully reflects the rational homotopy type of $X$ under mild hypotheses. The cohomology of this algebra is isomorphic to the cohomology of $X$.

We shall now explain this construction with more details. We start with the construction of the Sullivan algebra of polynomial forms. For every chain complex $V$, we denote by $S(V)$ the associated \emph{symmetric algebra} defined by 
\[S(V) \coloneqq \bigoplus_{r \geqslant  0} \left(V^{\otimes r}\right)_{\mathbb{S}_r} \ ,\]
and equipped with the only differential extending the differential of $V$ and compatible with the Leibniz rule.

\begin{definition}
For $n\geq 0$, let $K_n^*$ be the cochain complex over $\QQ$ defined by
\[K_n^*:=\bigoplus_{i=0}^n\QQ t_i\to \bigoplus_{i=0}^n\QQ dt_i\to0\to\ldots,\]
with the differential being obvious from the notations. The $\mathrm{CDGA}$ of polynomial forms on $\Delta^n$ is defined as 
\[\Omega^*_{PL}\left(\Delta^n\right) \coloneqq \frac{S(K_n^*)}{\langle\sum t_i-1 \rangle}.\] This induces a simplicial object in the category of CDGAs. This definition extends to a functor of polynomial forms
 \[ 
\begin{array}{cccl}
	\Omega^*_{PL}:& \mathrm{sSet} & \longrightarrow&  \mathrm{CDGA}^{op} \  \\
& 	X & \longmapsto & \Hom_{\mathrm{sSet}}\left(X, \Omega^*_{PL}\left(\Delta^{\bullet} \right) \right)
\end{array}
  \] 
This functor is the left adjoint in an adjunction
  \[\Omega^*_{PL}:\mathrm{sSet}\leftrightarrows \mathrm{CDGA}^{op}:\langle -\rangle \] where the right adjoint is then simply given by the formula
  \[\langle A\rangle_n\coloneqq\Hom_{\mathrm{CDGA}}(A,\Omega^*_{PL}(\Delta^n)) \ .\] Given a topological space $X$, one defines \[\Omega^*_{PL}(X) \coloneqq \Omega^*_{PL}(S_{\bullet}(X)) \] where $S_{\bullet}(X) : = \Hom_{\mathrm{Top}}(\Delta^{\bullet},X)$ is the singular simplicial set associated to $X$. The fundamental theorems of rational homotopy theory states that the induced functor from the rational homotopy category of simplicial sets
  \[\Omega^*_{PL}:\mathrm{Ho(sSet)}_{\mathbb{Q}}\to \mathrm{Ho(CDGA)^{op}}\]
  is fully faithful when restricted to nilpotent simplicial of finite type.
\end{definition}

\begin{definition}\label{definition : formality of spaces}
A topological space $X$ is formal, if $\Omega^*_{PL}(X)$ is related to $H^*(X;\QQ)$ by a zig-zag of quasi-isomorphisms of $\mathrm{CDGA}$s. 
\end{definition}

It follows that, if $X$ is formal, one can reconstruct the rational homotopy type of $X$ simply from the datum of the cohomology algebra of $X$. In particular, if $X$ is simply connected one gets the following formula for homotopy groups
\[\pi_i(X)\otimes\mathbb{Q}\cong \pi_i\langle  Q(H^*(X,\mathbb{Q}))\rangle\]
where $Q(H^*(X,\mathbb{Q}))$ denotes a cofibrant replacement of $H^*(X;\mathbb{Q})$ in the model category of $\mathrm{CDGA}$s. It follows from this discussion that the determination of rational homotopy groups of a formal space becomes a purely algebraic computation.
	
\begin{example} $\leavevmode$
		\begin{enumerate}
			\item The sphere $\mathbb{S}^n$ is formal, for all $n \geqslant 1$. Its cohomology is given by \[H^*\left(\mathbb{S}^n; \mathbb{Q} \right) \cong \mathbb{Q} [x]/x^2 \ , \] where $x$ is such that $|x| = n$. Let us introduce the following CDGA \[ \mathcal{M}_n =  \left\{
			\begin{array}{l l l  l}
				\QQ[u], & \begin{array}{l}
					|u| = n  
				\end{array} & \begin{array}{l} d = 0 \end{array} & \mbox{if } n \mbox{ is odd} \\
				\QQ[u,v], &  \begin{array}{l}
					|u| = n  \\  |v| = 2n -1
				\end{array}   &  \begin{array}{l} du = 0 \\ dv = u^2 \end{array}  & \mbox{if } n \mbox{ is even}
			\end{array}
			\right.  \] 
There is a zig-zag of quasi-isomorphisms 
\[ \Omega^*_{PL}(\mathbb{S}^n)  \; \underset{\sim}{\overset{f}{\longleftarrow}} \;  \mathcal{M}_n \;   \underset{\sim}{\overset{g}{\longrightarrow}}  \; H^{*}(\mathbb{S}^n; \QQ) \ , \]
where $f$ and $g$ are defined as follows. Let us set $g (u) = x$ for all $n$ and $g(v) = 0$ for $n$ even. Since the cohomology of $\Omega^*_{PL}(\mathbb{S}^n)$ is given by $H^{*}(\mathbb{S}^n; \QQ)$, there exist a cocycle $\tilde{u} \in \Omega^n_{PL}(\mathbb{S}^n)$ such that $[\tilde{u}]  = x$ for all $n$ and  $\tilde{v} \in \Omega_{PL}^{2n-1}(\mathbb{S}^n)$ such that $d\tilde{v}  = \tilde{u}^2$ in the case where $n$ even. One defines a quasi-isomorphism $f$ by setting $f(u) =\tilde{u}$ and $f(v) =\tilde{v}$. 
			Although the integral homotopy groups of spheres are still mostly unknown, the previous result provides a method to efficiently calculate their rational homotopy groups. If $n$ is odd, this leads to 
\[ \pi_*(S^n)\otimes\mathbb{Q} \cong  \left\{
			\begin{array}{c l }
			\mathbb{Q}  &  \mbox{if } *=n \\
				0 &  \mbox{otherwise \ ,}
			\end{array}
			\right.  \] 
and if $n$ is even, we have
\[ \pi_*(S^n)\otimes\mathbb{Q} \cong  \left\{
			\begin{array}{c l }
			\mathbb{Q}  &  \mbox{if } *=n,2n-1 \\
				0 &  \mbox{otherwise \ .}
			\end{array}
			\right.  \]

			\item Complex projective space $\mathbb{C P}^n$, for all $n \geqslant 0$ are formal. The proof is very similar to the proof for even spheres.
			
			\item Lie groups are formal. The proof is similar to the one of spheres and generalizes to any simply connected spaces $X$ whose rational cohomology is free as a graded algebra.
	
		\end{enumerate}
\end{example}

\section{\textcolor{bordeau}{The example of compact Kähler manifolds}}\label{2}

There is a long tradition of using Hodge theory as a tool for proving formality results. This section focuses on the first result in this direction: the formality of compact Kähler manifolds established by P.\ Deligne, P.\ Griffiths, J.\ Morgan and D.\ Sullivan in \cite{DGMS75}. This result can be stated as formality of some functors, through two theorems that are very much related: a contravariant version and a covariant one.

\subsection{The contravariant version}

Recall that a compact Kähler manifolds is a manifold with three mutually compatible structures: a complex structure, a Riemannian structure, and a symplectic structure. One important source of examples is given by smooth projective complex varieties. Those are K\"ahler by pulling back the K\"ahler structure of the complex projective space in which they embed. Complex projective spaces have a K\"ahler structure given by the Fubini-Study metric, see \cite[Section 3.3.2]{Voi02}. Let $\K$ be the category of compact Kähler manifolds.

\begin{theorem}[{\cite[Main theorem]{DGMS75}}]\label{contravariant}
	The functor of differential forms $$\mathcal{E}^* : \K^{op} \to \ch^*(\mathbb{R})$$ is a formal symmetric monoidal functor. 
\end{theorem}

\begin{proof}
Let $M$ be a compact Kähler manifold. Since $M$ is a complex manifold, its tangent bundle is equipped with an endomorphism $J : TM \to TM$ satisfying $J^2 = -id$. By dualizing $J$, it induces an endomorphism of the cotangent vector bundle and therefore an automorphism of the de Rham complex $\mathcal{E}^*(M)$. Thus, this complex is equipped with its usual differential denoted $d$, but also with another operator called $d^c$ defined by $d^c = -J d J$. These operators make $\mathcal{E}^*(X)$ into a bicomplex that moreover satisfies a lemma, called $dd^c$-lemma at the heart of the proof.

\begin{lemma}[{\cite[$dd^c$-lemma]{DGMS75}}]
If $x$ is a differential form such that $dx = 0$ and $x = d^c y$, then $x = dd^c(z)$ for some $z$. 
\end{lemma}

Let $\mathcal{E}^*(M)$ be the real de Rham complex of $M$, $^c \mathcal{E}^*(M)$ be the subcomplex of $d^c$-closed forms, and $H^*_{d^c}(M)$ be the quotient complex $^c \mathcal{E}^*(M) /d^c(\mathcal{E}^*(M))$. Then we have a diagram of the form
\[(\mathcal{E}^*(M),d) \overset{i}{\longleftarrow} (^c\mathcal{E}^*(M),d) \overset{\pi}{\longrightarrow} (H^*_{d^c}(M),d)\]
where $i$ corresponds to the inclusion of the subcomplex and $\pi$ to the quotient. Let us first prove that $i$ induces an isomorphism in cohomology. For all \[[x] \in H^*(\mathcal{E}^*(M),d) \ ,\] the form $d^cx$ satisfies the hypothesis of the $dd^c$-lemma. Thus, there exists an element $y$ such that $d^c x = d d^c y$. Setting $z = x + dy$, we get $d^c(z) = 0$ and $i$ induces a surjection in cohomology. Let $y \in {}^c\mathcal{E}^*(M)$ be a closed form which is exact in $\mathcal{E}^*(M)$. Then, we have $d^c y = 0 = dy$ and $y = dz$. Thus, there exists $w$ such that $y = dd^c w$ and $y$ is necessarily trivial. This shows that $i_*$ is injective and thus, an isomorphism. Let us now prove that $\pi$ induces an isomorphism in cohomology. For all \[[y] \in H^*(^c\mathcal{E}^*(M),d) \ , \] the element $y$ is $d^c$-closed. Thus $dy$ satisfies the hypothesis of the $dd^c$-lemma. There exists $z$ such that $dy = d d^c z$. Setting $ x = y + d^c z$, then $dx = 0$ and $[x] = [y]$. Thus, $\pi$ is surjective. Finally, let $y$ be such that $[y] = 0$ in $(H^*_{d^c}(M),d)$. Then $y = d^c(w)$ and by the $dd^c$-lemma, there exists $z$ such that $y = d d^c z$ and $\pi$ is injective. 

\noindent Note that the differential induced by $d$ on $H^*_{d^c}(M)$ is $0$. Indeed, if $d^c y = 0$, then by the $dd^c$-lemma, there exists $w$ such that $dy = dd^c w$ and $dy \in \mathrm{Im} (d^c)$. We deduce that $[dy] = 0$ in $H^*_{d^c}(M)$. Thus, there exists an isomorphism 
\[H^*_{d^c}(M) \cong H^*(\mathcal{E}^*(M)) \;.\] 
Furthermore, since $d^c$ satisfies the Leibniz rule, $^c\mathcal{E}^*$ inherits a monoidal symmetric functor structure from that of $\mathcal{E}^*$. Finally, since the morphisms $i$ and $\pi$ are natural and compatible with the structure of monoidal symmetric functors, we conclude that $\mathcal{E}^*$ is formal.
\end{proof}

The functor of de Rham forms $\mathcal{E}^*(-)$ on smooth differentiable manifolds is naturally quasi-isomorphic to the functor $\Omega^*_{PL}(-)\otimes_{\mathbb{Q}}\mathbb{R}$, see \cite[Corollary 9.9]{GM}. This leads to the following corollary.

\begin{corollary}[{\cite[Corollary 1]{DGMS75}}]\label{coro : DGMS}
	Let $M$ a compact Kähler manifold. The real homotopy type of $M$, that is the homotopy type of the real CDGA \[\Omega^*_{PL}(M)\otimes_{\mathbb{Q}}\mathbb{R} \ ,\] is determined by the real cohomology algebra of $M$. In particular, if $M$ is pointed and simply connected, there is an isomorphism
\[\pi_*(M) \otimes_{\mathbb{Z}} \mathbb{R} \cong \pi_*\langle Q(H^*(M, \mathbb{R})) \rangle\]
where $\langle-\rangle$ is the real version of Sullivan realization functor and $Q$ denotes a cofibrant replacement in the model category of CDGAs.
\end{corollary}

In light of Definition \ref{definition : formality of spaces}, it is very natural to wonder to what extent formality depends on the coefficient ring. A formality result with coefficients in a certain field $\mathbb{K}$ will also be satisfied for any extension $\mathbb{L}$ of this field. Formality descent gives a partial converse to this result.

\begin{theorem}[Formality descent]\label{descent}
Let $\mathbb{L}\subset\mathbb{K}$ be two characteristic zero fields. Let $A$ be a $\mathrm{CDGA}$ over $\mathbb{L}$ with finite type cohomology. The algebra $A$ is formal if and only if $A\otimes_{\mathbb{L}} \mathbb{K}$ is formal. 
\end{theorem}

\begin{proof}
This theorem is proved in \cite[Corollary 6.9]{HJ79}. See also Theorem \ref{descent2} in these notes for a somewhat different point of view.
\end{proof}

\begin{remark}
Two $\mathrm{CDGA}$s over $\mathbb{L}$ may become quasi-isomorphic after extending the scalars to $\mathbb{K}$ without being quasi-isomorphic over $\mathbb{K}$. As a simple example, one can take the following two commutative algebras over $\mathbb{R}$~, \[A=\mathbb{C} \quad \mbox{and} \quad B=\mathbb{R}\times\mathbb{R} \ .\] These two algebras are not isomorphic though we have isomorphisms
\[A\otimes_{\mathbb{R}}\mathbb{C}\cong \mathbb{C}\times\mathbb{C}\cong B\otimes_{\mathbb{R}}\mathbb{C}.\]
\end{remark}

Using Theorem \ref{descent} and Corollary \ref{coro : DGMS}, one deduces the following result.

\begin{corollary}\label{formality over Q}
Let $M$ a compact Kähler manifold. The rational homotopy type of $M$ is determined by its cohomology as commutative graded algebras. In particular, if $M$ is simply connected, there is an isomorphism \[\pi_*(M) \otimes_{\mathbb{Z}} \mathbb{Q} \cong \pi_*\langle Q(H^*(M, \mathbb{Q})) \rangle \ .\]
\end{corollary}

\begin{remark}
One issue with this corollary is that we have lost functoriality in the process of descending formality. For instance, if $G$ is a discrete group acting on a compact K\"ahler manifold, it is not at all obvious that the isomorphism
\[\pi_*(M) \otimes_{\mathbb{Z}} \mathbb{Q} \cong \pi_*\langle Q(H^*(M, \mathbb{Q})) \rangle\]
is an isomorphism of representations of $G$. It is however be an isomorphism of $G$-representations after extending the scalars to $\mathbb{R}$ thanks to the functoriality of the Deligne--Griffiths--Morgan--Sullivan Theorem. As we shall see later in these notes, functoriality does hold over $\mathbb{Q}$ but requires a different argument (see Theorem \ref{sullivan}).
\end{remark}

\subsection{The covariant version}

In the paper \cite{GNPR05}, the authors elaborate on the method of \cite{DGMS75} and prove that operads (as well as cyclic operads, modular operads, etc.) internal to the category of compact Kähler manifolds are formal. They establish the following covariant version to Theorem \ref{contravariant}. 

\begin{theorem}[{\cite[Corollary~3.2.3]{GNPR05}}] \label{covariant}
	The functor of singular chains $$C_*(-;\mathbb{R}) : \K \to \ch_*(\mathbb{R})$$ is a formal lax symmetric monoidal functor. 
\end{theorem}

\begin{remark}
	The lax symmetric monoidal structure on the functor $C_*(-;\RR)$ comes from the shuffle product 
	\[ \delta_{X,Y} : C_*(X; \RR) \otimes  C_*(Y; \RR)  \to  C_*(X \times Y; \RR)\]
and the obvious unit morphism, see \cite[Theorem~5.2]{EilenbergMacLane53}. 
\end{remark}

Let $M$ be a differentiable manifold. Consider $\mathcal{E}^*(M)$, the complex of differential forms. We can make it into a locally convex topological vector space by giving it the topology of compact convergence for all derivatives of forms. We can then consider its topological dual, denoted $\mathcal{E}_*'(M)$ equipped with the strong topology. An element of $\mathcal{E}_*'(M)$ is called a de Rham current with compact support. We thus have a covariant functor \[\mathcal{E}_*' : \mathrm{Dif} \to \ch_*(\mathbb{R}) \ ,\] where $\mathrm{Dif}$ denotes the category of differentiable manifolds and smooth maps.  This functor has a symmetric monoidal structure inherited from the wedge product of differential forms. More precisely, given $S \in \mathcal{E}_*'(M)$ and $T \in \mathcal{E}_*'(N)$, we define $\kappa(S\otimes T)\in\mathcal{E}_*'(M\times N)$ by the formula
\[ \langle \kappa(S \otimes T), \pi_M^*(\omega) \wedge \pi^*_N(\nu) \rangle = \langle S, \omega \rangle \cdot \langle T, \nu \rangle\]
valid for all $\omega \in \mathcal{E}^*(M)$ and $\nu \in \mathcal{E}^*(N)$, where $\pi_M$ and $\pi_N$ denote the projections on $M$ and $N$ respectively. It turns out that this formula is sufficient for defining a current as the map
\[\begin{array}{ccl}
\mathcal{E}^*(M)\otimes\mathcal{E}^*(N)&\longrightarrow &\mathcal{E}^*(M\times N)\\(\omega,\nu)&\longmapsto& \pi_M^*(\omega) \wedge \pi^*_N(\nu)
\end{array}\] has dense image.

\begin{theorem}[{\cite[Proposition~2.4.1]{GNPR05}}]\label{faible} The two functors
\[C_* , \mathcal{E}_*' : \mathrm{Dif} \to \ch_*(\R)\]
	are weakly equivalent symmetric monoidal functors. 
\end{theorem}

\begin{proof}
	The proof is dual to de Rham's theorem giving a quasi-isomorphism between the de Rham complex $\mathcal{E}^*$ and the functor of singular cochains. For every differentiable manifold $M$, let $C^{\infty}_*(M; \mathbb{Z})$ be the sub-complex of singular chains generated by the $\mathcal{C}^{\infty}$-maps. The shuffle product of $\mathcal{C}^{\infty}$-singular chains is also a $\mathcal{C}^{\infty}$-singular chain. We obtain a symmetric monoidal functor 
\[C^{\infty}_* : \mathrm{Dif} \to \ch_*(\R) \ . \] 
Let $M$ be a differentiable manifold and let $c : \Delta^p \to M$ be a $\mathcal{C}^{\infty}$-singular simplex. By Stokes theorem, integration along $c$ induces a morphism 
	\[C^{\infty}_*(M) \to \mathcal{E}_*'(M), \quad c \mapsto \int_c\]
	defined by $\int_c \omega \coloneqq \int_{\Delta^p} c^*(\omega)$, see \cite[Chapter V. 5]{Bredon97}. De Rham theorem's implies that integration induces an isomorphism in homology. Thus \[\int : C^{\infty}_* \to \mathcal{E}_*'\] is a weak equivalence between functors, see \cite[Chapter V. Theorem 9.1]{Bredon97}. To conclude that this is a weak monoidal equivalence, one needs to verify that integration is compatible with the monoidal structure. We refer the reader to \cite[Proposition~2.4.1]{GNPR05} for more details about this. Furthermore, the natural inclusion of $\mathcal{C}^{\infty}$-singular chains in the singular chains defines a symmetric monoidal natural transformation $C^{\infty}_* \to C_*$, since the structures of symmetric monoidal functors are both defined with the shuffle product. Finally, the inclusion induces an isomorphism in homology, see for example \cite[Page~291]{Bredon97}. By composing the two weak equivalences, we get that $C_*$ and $\mathcal{E}_*'$ are weakly equivalent symmetric monoidal functors. 
\end{proof}

\begin{remark}
There is a small mistake in \cite{GNPR05}. Instead of $\mathcal{E}_*'$, the authors use the functor of de Rham currents $\mathcal{D}_*'$ which is the topological dual of the functor of de Rham differential forms with compact support. The issue is that this functor is not quasi-isomorphic to singular chains but instead computes Borel-Moore homology of the manifold. However, the restriction of the two functors $\mathcal{E}_*'$ and $\mathcal{D}_*'$ to compact manifolds are isomorphic and \cite{GNPR05} only applies the previous to compact manifolds. 
\end{remark}

Since $C_*$ and $\mathcal{E}_*'$ are two weakly equivalent symmetric monoidal functors, Theorem \ref{covariant} is a consequence of the following.

\begin{theorem}\label{theo : DGMS}
	The functor of currents $\mathcal{E}_*' : \K \to \ch_*(\R)$ is a formal symmetric monoidal functor. 
\end{theorem}

\begin{proof}
	The proof is essentially the same than the one of Theorem \ref{contravariant}. One can prove that the Kähler identities between the operators $d,d^c, \Delta, \dots $ of the de Rham complex of differential forms are also satisfied by the corresponding dual operators on the de Rham complex of currents. Thus, we have a similar version of the $dd^c$-lemma that holds for the complex $\mathcal{E}_*'$. Let $M$ be a compact Kähler manifold. Let $^c\mathcal{E}'_*(M)$ be the subcomplex of $\mathcal{E}'_*(M)$ defined by the $d^c$-closed currents and let us denote 
	\[H_*^{d^c}(M) \coloneqq ^c\mathcal{E}'_*(M)/ d^c(\mathcal{E}'_*(M)) \ .\]
	As in the proof of Theorem \ref{contravariant}, we obtain a diagram
	\[(\mathcal{E}'_*(M),d) \overset{i}{\longleftarrow} (^c\mathcal{E}'_*(M),d) \overset{\pi}{\longrightarrow} (H_*^{d^c}(M),d)\ ,\] where $i$ corresponds to the inclusion and $\pi$ to the quotient. Both maps are weak equivalences and the differential induced by $d$ on the quotient is zero (using $dd^c$-lemma). Since the morphisms $i$ and $\pi$ are morphisms of lax monoidal functors, we conclude that $\mathcal{E}'_*$ is formal.   
\end{proof}

\begin{corollary}
	If $\mathcal{O}$ is an operad in $\K$ then $C_*(\mathcal{O}, \mathbb{R})$ is formal. 
\end{corollary}

\begin{proof}
	This follows directly from Theorem \ref{covariant} and Proposition \ref{ima}. 
\end{proof}

\begin{example}[Moduli spaces of stable algebraic curves of genus 0]\label{moduli}
Let \[\mathcal{O} = \{\overline{\mathcal{M}}_{0,l}\}_{l}\]  be the (cyclic) operad of moduli spaces of stable algebraic curves of genus 0, defined as follows. We denote by $\mathcal{M}_{0,l}$ the moduli space of $\ell$-tuples, $(x_1, \dots, x_{\ell})$, of distinct points of the complex projective line $\Co \mathbb{P}^1$ modulo projective automorphisms; that is the transformations of the form $$\Co \mathbb{P}^1 \to \Co \mathbb{P}^1, \quad [\xi_1, \xi_2] \mapsto [a\xi_1 + b \xi_2, c \xi_1 + d \xi_3] $$ for $a,b,c,d \in \Co$ such that $ad - bc \neq 0$. The space $\overline{\mathcal{M}}_{0,l}$, originally defined in \cite{Deligne1969}, corresponds to a certain compactification of $\mathcal{M}_{0,l}$ and is the moduli space of stable algebraic curves of genus $0$ with $\ell$-marked points. Briefly, a stable curve is a curve that can have nodal singularities but that has a finite group of automorphisms (this imposes that each component of the curve has at least $3$ points that are either nodal points or marked points). The operations of gluing two stable curves along marked points turns the collection of spaces $\overline{\mathcal{M}}_{0,l}$ into a cyclic operad $\mathcal{O}$ in the category of smooth projective varieties, so in particular in the category of compact K\"ahler manifolds. Using the previous theorem, one can conclude that $C_*(\mathcal{O}, \mathbb{R})$ is formal.
\end{example}

\begin{remark}
The above discussion generalizes to the modular operad obtained by taking moduli spaces of stable curves of all genera. The only issue is that the higher genus moduli spaces are not compact K\"ahler manifolds anymore. However, they are Deligne-Mumford stacks and the technology that we have just explained extends to this context. In \cite{GNPR05}, the authors are able to prove that this modular operad is formal.
\end{remark}

\section{\textcolor{bordeau}{Purity implies formality}}

The aim of this section is to give an equivalent characterization of formality of symmetric monoidal functors in terms of weight decompositions. We are going to prove that if a certain condition called \emph{purity} is satisfied, then formality is guaranteed.

\subsection{An equivalent definition of formality}

Let $\mathbf{k}$ be a field. We denote by $\mathrm{grVect}$ the category of graded vector spaces over $\mathbf{k}$. This abelian category inherits a symmetric monoidal structure from that of $\mathbf{k}$-vector spaces (the symmetry isomorphism does not involve any sign).

Every chain complex in $\mathrm{grVect}$ has two gradings. One comes from the underlying category $\mathrm{grVect}$ and is called ``weight''. We will denote it with a superscript. The other one corresponds to the homological grading. We will keep the term ``degree'' for this grading denoted with a subscript. The category $\ch_*(\mathrm{grVect})$ also inherits a symmetric monoidal structure given by
\[(C\otimes D)_m^p=\bigoplus_{n,q\in\mathbb{Z}}C_n^q\otimes D_{m-n}^{p-q}\]
The symmetry isomorphism involves the usual Koszul sign for the homological degree but not for the weight.

\begin{definition}
An object $V$ of $\ch_*(\mathrm{grVect})$ has \emph{pure homology} if
\[H_n(V)^p = 0 \quad \mbox{for all } p \neq n.\]
Denote by $\ch_*(\mathrm{grVect})^{pure}$ the full subcategory of $\ch_*(\mathrm{grVect})$ consisting of chain complexes with pure homology.
\end{definition}

The following proposition is straightforward exercise of linear algebra. A proof can be found in \cite[Proposition 2.7]{CH20}. Associated to Proposition \ref{composition}, it is the basis for many formality results that we will establish.

\begin{proposition}\label{forget}
The forgetful functor defined by forgetting the weight
\[U : \ch_*(grVect)^{pure} \to \ch_*(\mathbf{k})\]
is formal as lax symmetric monoidal functor.  
\end{proposition}

From this, one can derive the following formality criterion.

\begin{proposition}\label{factorisation}
A symmetric monoidal functor $F : \mathcal{C} \to \ch_*(\mathbf{k})$ is formal if and only if it is weakly equivalent to a functor $\tilde{F}$ that admits a factorisation,	
\begin{center}
\begin{tikzcd}
	&\ch_*(grVect)^{pure} \arrow[d, "U"] \\
	\mathcal{C} \arrow[r,"\tilde{F}"'] \arrow[ru, "G"] & \ch_*(\mathbf{k})
\end{tikzcd}
\end{center}

\end{proposition}

\begin{proof}
If $F$ is formal, it is weakly equivalent to the functor $H\circ F$ which splits as a direct sum $\oplus_n H_n\circ F$. This splitting satisfies the desired condition. Conversely, if $\tilde{F}$ admits such a factorization, then $\tilde{F}$ is formal using the previous proposition and Proposition \ref{composition}.
\end{proof}

The following Proposition gives a method for producing pure weight gradings.

\begin{proposition}\label{condformality} A symmetric monoidal functor $F : \mathcal{C} \to \ch_*(\mathbf{k})$ is formal if it is weakly equivalent to a functor $\tilde{F}$ such that 
	\begin{enumerate}
	\item $\tilde{F}$ is object-wise and degree-wise finite dimensional.
	\item $\tilde{F}$ has an endomorphism $\sigma$ which acts as multiplication by $\lambda^n$ on $H_n(\tilde{F}(c))$  for all $n \geqslant 0$ and $c\in Ob(\mathcal{C})$, where $\lambda \in \mathbf{k}$ is a unit of infinite order. 
	\end{enumerate} 
\end{proposition}

\begin{proof}
Since the property of being formal is stable under weak equivalences, it is sufficient to check that $\tilde{F}$ is formal. This functor admits a sub-complex given by the direct sum $$ i : \bigoplus_{n \in \mathbb{Z}} \tilde{F}^n  \lhook\joinrel\xrightarrow{\quad  \quad } \tilde{F}$$ where each $\tilde{F}^n$, for $n \geqslant  0$, is the corresponding generalized eigenspace for the eigenvalue $\lambda^n$. Note that this is a sub-complex since the differential has to preserve generalized eigenspaces. The inclusion $i$ is a lax symmetric monoidal functor and a quasi-isomorphism since the other generalized eigenspaces will not contribute to homology because of condition (2). Then, the sub-functor $$G \coloneqq \bigoplus_{n \in \mathbb{Z}} \tilde{F}^n$$ has a canonical weight grading by construction, and admits a factorisation as in Proposition \ref{factorisation}. This implies that $G$ and hence $\tilde{F}$ and $F$ are formal.
\end{proof}

\subsection{The formality of the little disks operad}\label{little disks operad}  The original proof of formality of the little disks operad $\mathcal{D}_2$ is due to Kontsevich \cite{kon99} and Tamarkin \cite{Tam98} independently. In this section, we present another proof relying on Proposition \ref{condformality} which is due to Petersen \cite{Petersen14}.

Let $\mathbf{PaB}$ be the operad parenthesized braids. This is the operad in groupoids such that $\mathbf{PaB}(n)$ are parenthesized permutations of $\{1, \dots n\}$ and morphisms are braids on $n$ strands maintaining the same label at the start and at the end of each strand. We refer the reader to \cite[Section~2.3]{Damien} in this volume for more details. We also consider its $\QQ$-pro-unipotent completion, denoted $\widehat{\mathbf{PaB}}_{\QQ} $. The $\QQ$-pro-algebraic Grothendieck-Teichmüller group is defined as
\[\widehat{GT}_{\QQ} \coloneqq \Aut^+_{\ensuremath{\mathrm{OpGrpd}}}\left(\widehat{\mathbf{PaB}}_{\QQ}\right)\; .\]
where $\Aut^+$ denotes the group of automorphisms in the category of operads in groupoids that induce the identity map on the objects of the groupoid in each arity. The operad $\mathbf{PaB}$ is weakly equivalent to the operad of fundamental groupoids of $\mathcal{D}_2$. Since $\mathcal{D}_2(n)$ is a $K(\pi,1)$-space, we obtain a weak equivalence
\[\mathcal{D}_2\simeq \mathrm{B}(\mathbf{PaB})\]
where $\mathrm{B}$ denotes the classifying space functor. It turns out that the configuration spaces are also rational $K(\pi,1)$, see \cite{papadimarational}. This implies that the weak equivalence above induces an equivalence
\[(\mathcal{D}_2)_\QQ\simeq \mathrm{B}\widehat{\mathbf{PaB}}_{\QQ} \ .\]
From this, we obtain a quasi-isomorphism of differential graded operads
\[C_*(\mathcal{D}_2; \QQ) \cong C_*\left(\mathrm{B}\widehat{\mathbf{PaB}}_{\QQ}, \QQ\right) \ .\] The natural action of $\widehat{GT}_{\QQ}$ on $\widehat{\mathbf{PaB}}_{\QQ}$ extends to an action on the right-hand side. Recall that the homology of the little two-disks operad is the Gerstenhaber operad. The induced action of $ \widehat{GT}_{\QQ}$ on $H_*(\mathcal{D}_2)$ is given \[\begin{array}{ccl}
	\widehat{GT}_{\QQ} \times  H_n(\mathcal{D}_2) &\longrightarrow &H_n(\mathcal{D}_2) \\ (\sigma, a )&\longmapsto& \sigma \cdot a = \chi(\sigma)^n a
\end{array}\]
where $\chi : \widehat{GT}_{\QQ} \to \QQ ^\times$ is the cyclotomic character. The map $\chi$ is surjective by \cite[Section~5]{Drinfeld90}. Let $\alpha \in \mathbb{Q}^{\times}$ of infinite order. By surjectivity, there is a lift \[\tilde{\alpha} \in \widehat{GT}_{\QQ}\] such that $\chi(\tilde{\alpha}) = \alpha$. This lift induces an endomorphism of \[C_*\left(\mathrm{B}\widehat{\mathbf{PaB}}_{\QQ}, \QQ\right)\] which acts by multiplication by $\alpha^n$ on the homology group of degree $n$ for all $n \in \NN$. The dg operad at stake is not finite dimensional in each degree and arity. Let us consider
\[M \overset{\sim}{\longrightarrow} C_*\left(\mathrm{B}\widehat{\mathbf{PaB}}_{\QQ}, \QQ\right)\] to
be a minimal model as chain operad. This can easily be shown to be finite dimensional so it satisfies the condition (1) of the previous proposition. This minimal model inherits an endomorphism $\tilde{\alpha}$ satisfying the condition (2) of the proposition \ref{condformality} and this leads to the desired formality result.

\begin{remark}
A similar argument was used by Boavida de Brito and the second author in order to prove formality of the higher dimensional little disks operads in \cite{BdBH21}. The original formality argument is also due to Kontsevich with a detailed proof by Lambrechts-Voli\'c (see \cite{kon99,lambrechtsformality}).
\end{remark}

\begin{remark}
We expand a little bit about a subtle point of this section. The fact that configuration spaces of points in $\mathbb{R}^2$ are rational $K(\pi,1)$-spaces is non-trivial. It can happen that the rationalization of a $K(\pi,1)$ space has non-zero higher homotopy groups. As an example, we can consider the classifying space of the infinite linear group of the integers $BGL_{\infty}(\mathbb{Z})$. This is of course a $K(\pi,1)$-space but, the canonical map to Quillen's plus construction
\[BGL_{\infty}(\mathbb{Z})\to BGL_{\infty}(\mathbb{Z})^+\]
induces an isomorphism in homology with coefficient in $\mathbb{Z}$ and hence also with coefficients in $\mathbb{Q}$. It follows that the two spaces have the same rational homotopy type. But the higher homotopy groups of the rationalization of $BGL_{\infty}(\mathbb{Z})^+$ are, by definition, the rationalization of the higher $K$-groups of $\mathbb{Z}$. Those are known to be non-trivial in degree $4k+1$ for each $k>0$, by work of Borel, see \cite[Proposition 12.2]{borelstable}.
\end{remark}

\section{\textcolor{bordeau}{Interlude: Infinity categories}}

The sections \ref{Hodge} and \ref{torsion} uses the formality criteria of Proposition \ref{factorisation} in order to deduce formality. To do so, it will be convenient to use the extra flexibility provided by working with $\infty$-categories. The aim of this section is to introduce some concepts related to this subject.

\subsection{Classical infinity categories}

\begin{definition}[Nerve of a category] \label{nerve}
	The \emph{nerve} of a given category $\C$ is the the simplicial set $\N(\C)$ defined by $$\N(\C)_n = \mathrm{Fun}([n], \C) \ .$$ This definition extends to a fully faithful functor $\N : \mathrm{Cat} \to \mathrm{sSet} \ .$
\end{definition}

\begin{example} The nerve of the category $[n]$ is the standard $n$-simplex, i.e. \[\N([n]) \cong \Delta^n \ ,\] for all $n \geqslant  0$. Recall that the $k$-th horn $\Lambda^n_k \subset \partial \Delta^n$ for $0 \leqslant  k \leqslant  n$, is obtained from $\partial\Delta^n$ by removing the $k$-th face $\partial_k\Delta^n $ 
\end{example}

\begin{definition}[Infinity categories]
A simplicial set $\C$ is an \emph{$\infty$-category} if every inner horn \[\Lambda^n_k \to \C \ , \] for $0 < k <n$ can be extended to an $n$-simplex $\Delta^n \to \C$. 
\end{definition} 

\begin{example}
The nerve $\N(\C)$ of a category $\C$ is an infinity category. Through its nerve, any category can be seen as an $\infty$-category. From this example, we see that the $0$-simplices of an $\infty$-category should be thought of as objects and the $1$-simplices as morphisms. In addition, every extension of an inner horn $\Lambda^2_1 \to \N(\C)$ to a two simplex $\Delta^2 \to \N(\C)$ corresponds to the composition of morphisms. The higher dimensional inner horn extensions witness higher coherence in the associativity of compositions.
\end{example}

\begin{definition}[Functors between infinity categories]
	A \emph{functor} between two $\infty$-categories $\C$ and $\D$ is a morphism of simplicial sets $\C\to\D$. We denote by $\mathrm{Fun}(\C,\D)$ the simplicial set of functors from $\C$ to $\D$. This is an $\infty$-category and we shall refer to it as the $\infty$-category of functors from $\C$ to $\D$.
\end{definition}

\begin{definition}[Homotopy category]
	Let $\C$ an $\infty$-category. Let $f,g : x \to y$ be two morphisms in $\C$. They are said \emph{homotopic} if there is a $2$-simplex $\sigma : \Delta^2 \to \C$ with boundary $\partial \sigma $ corresponding to $(g,f, id_x)$.  The homotopy relation is an equivalence relation on the set of $1$-simplices with faces $x$ and $y$. Furthermore, there is an ordinary category $\mathrm{Ho}(\C)$, the \emph{homotopy category} of $\C$, with the same objects as $\C$ and morphisms the homotopy classes of morphisms in $\C$. We say that a $1$-simplex of an $\infty$-category is an \emph{equivalence} if it induces an isomorphism in the homotopy category.
\end{definition}

\begin{remark}
This homotopy category $\mathrm{Ho}(\C)$ does not capture all the homotopical information contained in $\C$. This can be enriched by constructing a \emph{mapping simplicial set} (usually called mapping space) denoted $\mathrm{map}_{\C}(x,y)$ for each pair of objects $x$ and $y$. These mapping spaces can be composed in a coherent homotopy associative way. The homotopy category of $\C$ is then simply the category obtained by applying $\pi_0$ to each of the mapping spaces. 
\end{remark}

\begin{definition}
A functor $f:\C\to \D$ between $\infty$-categories is called an \emph{equivalence} if it induces an equivalence of categories \[\mathrm{Ho}(\C)\to\mathrm{Ho}(\D)\] and weak equivalences of simplicial sets
\[\mathrm{map}_{\C}(x,y)\to \mathrm{map}_{\D}(f(x),f(y))\]
for any choice of $x$ and $y$ two objects of $\C$.
\end{definition}

\begin{definition}[{\cite[Definition~1.3.4.1]{Lurie12}}]
	Let $\C$ and $\mathcal{D}$ be two $\infty$-categories and $W$ be a collection of morphisms in $\C$. We say that a morphism $f : \C \to \mathcal{D}$ \emph{exhibit $\mathcal{D}$ as the localization of $\mathcal{C}$ with respects to $W$} if, for every $\infty$-category $\mathcal{E}$, the composition with $f$ induces a fully faithful embedding 
\[\mathrm{Fun}(\mathcal{D}, \mathcal{E}) \to \mathrm{Fun}(\C, \mathcal{E}),\]
whose essential image is the collection of functors $F: \C \to \mathcal{E}$ which carry each morphism in $W$ to an equivalence in $\mathcal{E}$. In this case, the $\infty$-category $\mathcal{D}$ is determined uniquely up to equivalence by $\C$ and $W$, and is denoted by \[\C[W^{-1}] \ .\] If $\C$ is an ordinary category, we denote this localization by $N_W(\mathcal{C})$. 
\end{definition}

Our main tool for constructing $\infty$-categories will be the following theorem.

\begin{theorem}[\cite{hinichdwyer}] Let $\C$ be an $\infty$-category and $W$ be a collection of morphism in $\C$. The $\infty$-category $\C[W^{-1}]$ exists. In particular, when $\C$ is an ordinary category, the homotopy category $\mathrm{Ho}(N_W(\mathcal{C}))$ is the one-categorical localization of $\C$ with respect to the maps of $W$.
\end{theorem}

\subsection{Symmetric monoidal $\infty$-categories}

Let $\mathbf{Fin}_*$ be the category of based finite sets. For every object $I$ of $\mathbf{Fin}_*$, we may choose a pointed bijection \[ \alpha : I \cong \langle n \rangle \coloneqq \{*, 1, \cdots,n\}\ , \] where $n+1$ is the cardinality of $I$. Using these isomorphisms, we may identify $\mathbf{Fin}_*$ with its full subcategory spanned by the objects $\langle n\rangle$ and we shall do so implicitly. For every pair of integers $1 \leqslant  i \leqslant  n$, we let $\rho^i : \langle n \rangle \to \langle 1 \rangle $ denote the morphism given by $\rho^i(j) = 1$ if $i = j$ and $*$ otherwise. 

\begin{definition}
	A \emph{symmetric monoidal $\infty$-category} is a coCartesian fibration of simplicial sets $p : \C^{\otimes} \to \N (\mathbf{Fin}_*)$ satisfying the following Segal condition. Let $\C^{\otimes}_{\langle n \rangle }$ denote the fiber of $p$ over the object $\langle n \rangle$. For each $n \geqslant  0$, the maps \[\{\rho^i : \langle n \rangle \to \langle 1 \rangle \}_{1 \leqslant  i \leqslant  n}\] induce functors $\rho^i_{!} : \C^{\otimes}_{\langle n \rangle } \to \C^{\otimes}_{\langle 1 \rangle } $ assembling into a map
	\[\C^{\otimes}_{\langle n \rangle } \to (\C^{\otimes}_{\langle 1 \rangle })^n.\]
	The Segal condition is requiring that this map is an equivalence of $\infty$-categories.
\end{definition}

\begin{definition}
A \emph{strong symmetric monoidal functor} between $\infty$-categories is a simply a morphism of coCartesian fibrations. A \emph{lax symmetric monoidal functor} is a map that is only required to preserve certain cocartesian morphisms (the so-called ``inert morphisms'').
\end{definition}

Our two main examples are the following.

\begin{example}\label{infinySMC}
	Let $(\C, \otimes, \mathbf{1})$ be a symmetric monoidal category. 
	\begin{enumerate}
		\item In \cite[Construction~2.0.0.1]{Lurie12} Lurie constructs a symmetric monoidal $\infty$-category
		\[\N(\C)^{\otimes}\to\N(\mathbf{Fin}_*)\]
		whose fiber $\N(\C)_{\langle 1\rangle}$ is identified with $\N(\C)$. This way, ordinary symmetric monoidal categories can be seen as symmetric monoidal $\infty$-categories. 
		\item Suppose that $\C$ is equipped with a collection of weak equivalences $W$. Assume that for all $X \in \C$ all $f : A \to B \in W$, the induced map $X \otimes A \to X \otimes B$ is in $W$, then $\N_W \C$ inherits a structure of symmetric monoidal $\infty$-category (see \cite[Proposition 3.2.2]{Hinich15}) that we shall denote by $N_W(\C)^{\otimes}$.
	\end{enumerate}
\end{example}

\begin{definition}
Let $(\mathcal{A}, \otimes, \mathbf{1})$ be an abelian symmetric monoidal category with infinite direct sums. We denote by $\mathbf{Ch}_*(\mathcal{A})$ the $\infty$-category obtained from $\ch_*(\mathcal{A})$ by inverting the quasi-isomorphisms~:
\[\mathbf{Ch}_*(\mathcal{A}) \coloneqq \N_{W} \ch_*(\mathcal{A}),\]
where $W$ is the class of all quasi-isomorphisms. 
\end{definition}

Assume now that the tensor product of $\mathcal{A}$ is exact in each variables. Let $X \in \ch_*(\mathcal{A})$ and $f : A \to B$ a quasi-isomorphism. Then the induced map $X \otimes A \to X \otimes B$ is also a quasi-isomorphism. Example \ref{infinySMC} implies that $\mathbf{Ch}_*(\mathcal{A})$ inherits a structure of symmetric monoidal $\infty$-category that we shall denote $\mathbf{Ch}_*(\mathcal{A})^{\otimes}$.

\begin{definition}\label{formal infinity functors}
Let $\C$ be a symmetric monoidal category and \[F : \N(\C) \to \mathbf{Ch}_*(\mathcal{A})\] a lax symmetric monoidal functor in the $\infty$-categorical sense. We say that $F$ is a \emph{formal symmetric monoidal $\infty$-functor} if $F$ and $H \circ F$ are equivalent in the $\infty$-category of  symmetric monoidal functors from $\N(\C)$ to $\mathbf{Ch}_*(\mathcal{A})$. 
\end{definition}

\begin{remark}
Clearly, a formal symmetric monoidal functor $\C \to \ch_*(\mathcal{A})$ induces a formal symmetric monoidal $\infty$-functor $\N(\C) \to \mathbf{Ch}_*(\mathcal{A})$. The following theorem due to Hinich gives a partial converse. 
\end{remark}

\begin{theorem}[\cite{Hinich15}]
Let $\C$ be a small symmetric monoidal category and let $\mathbf{k}$ be a characteristic zero field. If two symmetric monoidal functors \[F, G : \C \to \ch_*(\mathbf{k})\] are equivalent as symmetric monoidal $\infty$-functors $\N(\C) \to \mathbf{Ch}_*(\mathbf{k})$, they are weakly equivalent as symmetric monoidal functors.   
\end{theorem}

\begin{corollary}\label{infinity functors}
Let $\mathbf{k}$ be a characteristic zero field. Let $\C$ be a small symmetric monoidal category. Let $F :\C \to \ch_*(\mathbf{k})$ be a symmetric monoidal functor. If $F$ is formal as a symmetric monoidal $\infty$-functor $\N(\C) \to \mathbf{Ch}_*(\mathbf{k})$, then $F$ is formal as a symmetric monoidal functor.  
\end{corollary}

\section{\textcolor{bordeau}{Mixed Hodge structures}}\label{Hodge}

In paper \cite{CH20}, Cirici and the second author use mixed Hodge theory to produce decompositions for the singular chains functor and dually for Sullivan's functor in order to deduce formality. The purpose of this section is to explain these results. We denote by $\mathrm{Var}_{\Co}$ the category of complex schemes that are reduced, separated and of finite type. We will use the word variety for an object of this category.

\subsection{The definition of Mixed Hodge structures}

 We start by recalling what mixed Hodge structures are as well as their properties. 

\begin{definition}[Pure Hodge structures]
A \emph{pure Hodge structure} over $\mathbb{Q}$ of weight $n$ is a finite dimensional $\QQ$-vector space $V$ together with a decomposition of its complexification into a finite direct sum of complex subspaces
\[V \otimes_{\QQ} \Co = \bigoplus_{p=-\infty}^{+\infty} U_{p,n-p}, \quad \mbox{such that } \overline{U}_{p,n-p}= U_{n-p,p} \;.\]
\end{definition}

\begin{remark}
An equivalent definition is obtained by replacing the direct sum decomposition of $V \otimes_{\QQ} \Co$ with the \emph{Hodge filtration}, a finite decreasing filtration of $H : = V \otimes_{\QQ} \Co$ by complex subspaces $F^p H$ for $p \in \mathbb{Z}$, subject to the condition 
\[F^p H \oplus \overline{F^{n+1-p} H} = H\]
for all $p \in \mathbb{Z}$. The relation between these two descriptions is given by  $$U_{p,q} = F^p H \cap \overline{F^q H}\quad \mbox{and} \quad F^p H = \bigoplus_{i \geqslant  p} U_{i,n-i}\ . $$   
\end{remark}

\begin{definition}[Mixed Hodge structures]
A \emph{mixed Hodge structure} on a finite dimensional $\QQ$-vector space $V$ is given by 
\begin{enumerate}
	\item a finite increasing filtration $W$ of $V$, called the weight filtration; 
	\item a decreasing filtration $F$ on $H \coloneqq V \otimes_{\QQ} \Co$, called the Hodge filtration; 
\end{enumerate}
such that for all $m \geqslant  0$, each $\mathbb{Q}$-vector space \[Gr^m_W V \coloneqq W_m V/ W_{m-1} V \] is a pure Hodge structure with respect to the filtration induced by $F$ on \[Gr^m_W V \otimes_{\QQ} \Co \ .\] Morphisms of mixed Hodge structures are given by morphisms $f : V \to V'$ of $\QQ$-vector spaces that are compatible with filtrations. We denote by $\mathrm{MHS}_{\QQ}$ the category of mixed Hodge structures over $\QQ$.  
\end{definition}

\begin{remark}
Mixed Hodge structures form an abelian category by \cite[Theorem~2.3.5]{Deligne71}. Deligne shows that the morphisms are necessarily strictly compatible with both filtrations. The kernels and cokernels in this category coincide with the usual kernels and cokernels in the category of vector spaces, with the induced filtrations. Moreover, there is a symmetric monoidal structure on the category of mixed Hodge structures given by the usual tensor product of underlying vector spaces equipped with the induced filtrations.
\end{remark}

The theory of mixed Hodge structures on the cohomology of algebraic varieties was introduced by Deligne in 1970s (see \cite{Deligne1, Deligne71, Deligne74}).

\begin{theorem}[Deligne]\label{deligne}
Let $X$ be an algebraic variety over $\Co$ and $n \geqslant  0$. 
\begin{enumerate}
	\item The cohomology group $H^n(X; \QQ)$ carries a canonical mixed Hodge structure.
	\item This structure is functorial and compatible with the Künneth isomorphism.
	\item If $X$ is smooth and proper, then the mixed Hodge structure of $H^n(X; \QQ)$ is pure of weight $n$. 
\end{enumerate}
\end{theorem}

\begin{proof}[Ideas of the proof]
First, the cohomology with complex coefficients is given by the hypercohomology of the holomorphic de Rham complex~:
\[H^n(X;\mathbb{C})=\mathbb{H}^n(X;\Omega^*) \ , \] see \cite{Grothendieck}. 
In the smooth and proper case, the Hodge filtration is then the filtration induced by the so-called stupid filtration on this complex of sheaves,
\[F^pH^n(X;\mathbb{C})=\mathrm{im}\left(\mathbb{H}^n\left(X;\Omega^{*\geqslant  p}\right)\to \mathbb{H}^n\left(X;\Omega^*\right)\right) \ .\] 
It can then be shown using Hodge theory that this filtration is indeed a pure Hodge structure of weight $n$. In the smooth case, we can proceed as follows for the weight filtration. Using Nagata's compactification theorem and Hironaka's theorem on resolution of singularities, the variety $X$ can be embedded in a smooth complete variety $\overline{X}$ so that the complement $\overline{X}-X$ is a normal crossing divisor. This means that the inclusion $X \subseteq \overline{X}$ is locally a union of coordinate hyperplanes in $\Co^n$. Considering the Leray-Serre spectral sequence associated to the inclusion $X \subset \overline{X}$, we obtain a spectral sequence which will converge to the cohomology of $X$. The $E^1$-page of this spectral sequence looks as follows
\[E^{-s,t}_1 = H^{t-2s}\left(D^{(s)}, \QQ\right) \] where $D^{(s)}$ is the disjoint union of all $s$-fold intersections of components of the divisor. The groups $H^{t-2s}(D^{(s)}; \QQ)$ correspond to the cohomology of smooth and proper varieties so they carry pure Hodge structures. As with any spectral sequence, we get a filtration on the target whose associated graded is the page $E_{\infty}$ of the spectral sequence. In this particular case, the filtration on $H^n(X; \QQ)$ is, up to a shift, the weight filtration. Note that the associated graded of this filtration will be sub-quotients of cohomology groups of smooth and proper varieties so they carry pure Hodge structures. The functoriality of this construction is not obvious since the compactifications are not functorial. However, it turns out that the $E^2$ page of the Leray-Serre spectral sequence is natural (although it is not the case for $E^1$).  
\end{proof}

Let $\mathrm{FVect}_{K}$ be the category of filtered vector spaces over a certain field $K$. The weight filtration induces a functor \[W : \mathrm{MHS}_{\mathbb{Q}} \longrightarrow \mathrm{FVect}_{\QQ} \ .\]

\begin{theorem}[{\cite[Lemma 4.4]{CH20}}]\label{Splitting sur Q}
The weight filtration of $\mathrm{MHS}_{\QQ}$ naturally splits over $\QQ$ as a strong symmetric monoidal functor, i.e. the diagram 
\begin{center}
\begin{tikzcd}
&\mathrm{grVect}_{\QQ}\arrow[d,"T"]\\
\mathrm{MHS}_{\mathbb{Q}}\arrow[r, "W"'] \arrow[ru, "G"] & \mathrm{FVect}_{\QQ}
\end{tikzcd}
\end{center}
where $G\left(V\right)^p = Gr^W_p\left(V\right)$ and $T$ is the totalization functor
\[W^p\left(T\left(V\right)\right)=\oplus_{i\leqslant  p}V^i \ ,\] commutes up to a natural isomorphism.
\end{theorem}

\begin{proof}
This theorem was proved over the complex numbers by Deligne \cite[1.2.11]{Deligne71}. He constructs a functor $G_{\Co} : \mathrm{MHS}_{\QQ} \to \mathrm{grVect}_{\Co}$ defined by \[G_{\Co}\left(V\right)^p = Gr^W_p\left(V \otimes_{\QQ}\Co \right)\] that makes the following diagram commute
\begin{center}
\begin{tikzcd}
&\mathrm{grVect}_{\Co}\arrow[d,"T"]\\
\mathrm{MHS}_{\mathbb{Q}}\arrow[r,"W"'] \arrow[ru, "G_{\Co}"] & \mathrm{FVect}_{\Co}
\end{tikzcd}
\end{center}
where $T$ denotes the totalization functor. As proved in \cite{CH20}, this result can be descended to $\QQ$ using the fact that the set of such splittings form a torsor over a pro-unipotent algebraic group. Since this torsor has a $\Co$-point (given by Deligne's splitting) we obtain that it has a $\QQ$-point.
\end{proof}

\begin{remark}
It is important to emphasize that the construction of this splitting over $\QQ$ is through obstruction theory and does not give a formula for the splitting (contrary to Deligne's formula over $\Co$).
\end{remark}

\begin{remark}
The reasoning used in the above is strongly reminiscent of the proof of existence of Drinfeld's'd associators over $\QQ$. In that case one can explicitly construct associators over $\Co$ using the Knizhnik–Zamolodchikov equation. Since the set of associators is a torsor over a pro-unipotent group scheme, there must exists associators over $\QQ$ (see \cite{Drinfeld90}). 
\end{remark}

\subsection{Purity}

Let $\mathcal{A}$ be a symmetric monoidal abelian category. Denote by gr$\mathcal{A}$ the category of graded objects of $\mathcal{A}$ which inherits a symmetric monoidal structure as before. Let $\alpha$ be a rational number. We denote by \[\ch_*(gr\mathcal{A})^{\alpha \mbox{-}pure}\] the full subcategory of $\ch_*(gr\mathcal{A})$ spanned by those graded complexes $V = \bigoplus V^p_n$ with \emph{$\alpha$-pure homology}, i.e. such that $$H_n(V)^p = 0 \quad \mbox{for all } p \neq \alpha n.$$

Proposition \ref{forget} can be generalized as follows.

\begin{proposition}[{\cite[Proposition~2.7.]{CH20}}]\label{forget général}
	Let $\alpha$ be a non-zero rational number. The forgetful functor defined by forgetting the degree \[\ch_*(gr\mathcal{A})^{\alpha \text{-} pure} \longrightarrow \ch_*(\mathcal{A}) \]  is a formal as lax symmetric monoidal functor.  
\end{proposition}

\begin{definition}[$\alpha$-pure variety]
A smooth algebraic variety $X$ over $\Co$ is called \emph{$\alpha \text{-}$pure} if $H^k(X ;\QQ)$ is a pure Hodge structure of weight $\alpha k$ if $\alpha k \in \mathbb{Z}$ and $0$ otherwise. 
\end{definition}

\begin{example}\leavevmode \label{examples of weight computations}
\begin{enumerate}
	\item Smooth proper algebraic varieties over $\Co$ are $1$-pure by Theorem \ref{deligne}. 
	\item The open variety $\Co^* \coloneqq \Co \backslash \{0\}$ is $2$-pure. Indeed, let us consider the standard covering of $\Co \mathbb{P}^1$ given by two copies of $\Co$ whose intersection is $\Co^*$. The Mayer-Vietoris long exact sequence gives an isomorphism, $$H^{n-1}\left(\Co^*; \QQ\right) \cong H^n\left(\Co \mathbb{P}^1; \QQ\right) \cong  \left\{ \begin{array}{lll}
		\QQ &  \mbox{if } n= 2 \\
		0 &  \mbox{otherwise } \\
	\end{array} \right. $$ 
 for all $n \geqslant  2$. In this situation, all the morphisms involved in the long exact sequence are morphisms of mixed Hodge structures. In particular, the isomorphisms above are isomorphisms of mixed Hodge structures. Since $\Co \mathbb{P}^1$ is proper and thus $1$-pure by the first example, $H^2\left(\Co \mathbb{P}^1; \QQ \right)$ is a pure Hodge structure of weight 2, and therefore, so is $H^1\left(\Co^*; \QQ\right)$. Since $H^0\left(\Co^*; \QQ \right)$ is always a pure Hodge structure of weight $0$ and the higher cohomology groups are trivial, the variety $\Co^*$ is $2$-pure. 
	
	\item The variety $\Co^d \backslash \{0\}$ is $2d/(2d-1)$-pure. This space is homotopy equivalent to $\mathbb{S}^{2d-1}$ and thus for all $n \geqslant  0$, 
	$$H^{n}(\Co^d \backslash \{0\}; \QQ)  \cong  \left\{ \begin{array}{lll}
		\QQ &  \mbox{if } n = 0 \mbox{ or } n = 2d-1 \\
		0 &  \mbox{otherwise. } \\
	\end{array} \right. $$ The cohomology group is degree $0$ is of weight $0$. It remains to compute the weight in degree $2d-1$. Consider the following fibration,  $$ \Co^* \longrightarrow \Co ^d \backslash \{0\} \longrightarrow \Co \mathbb{P}^{d-1} \; .$$ The cohomology of the middle term can be computed using the Leray-Serre spectral sequence that will again be compatible with the mixed Hodge structures. We get 
	\[H^*(\Co \mathbb{P}^{d-1}; \QQ) \otimes H^*(\Co^*; \QQ) \implies H^*(\Co^d \backslash \{0\}; \QQ) \ .\]
Moreover, we know that 
\[ H^n(\Co \mathbb{P}^{d-1}; \QQ) \cong  \left\{ \begin{array}{lll}
	\QQ &  \mbox{if } n \mbox{ is even and } 0 \leqslant  n \leqslant  2(d-1)\\
	0 &  \mbox{otherwise .} 
\end{array} \right.
\] Computing the cohomology of $\Co^d \backslash \{0\}$ through the spectral sequence and by counting weight, we can show that $H^{2d-1}(\Co^d \backslash \{0\}; \QQ)$ is of weight $2d$ which implies that the variety $\Co^d \backslash \{0\}$ is $2d/(2d-1)$-pure. 
\end{enumerate}
\end{example}

To give another example, which generalizes example (3) above, we introduce the following definition.

\begin{definition}[Good arrangements]\label{defi : good arrangement}
Let $V$ be a finite dimensional $\Co$-vector space. We say that a finite set $\{H_i\}_{i \in I}$ of subspaces of $V$ is a \emph{good arrangement of codimension $d$ subspaces} if for any $J\subset I$, the intersection $\cap_{i\in J}H_i$ has codimension a multiple of $d$.
\end{definition}

\begin{proposition}\label{arrangements}
	Let $ \{H_1, \dots, H_k\}$ be a good arrangement of codimension $d$ subspaces of $\Co^n$. The algebraic variety $\Co^n - \cup_{i} H_i$ is $2d/(2d -1)$-pure.   
\end{proposition}

\begin{proof}
See \cite[Proposition 8.6]{CH20}. Note that the definition of good arrangement of codimension $d$ subspaces used in \cite{CH20} is not quite the same as the one above but it is easy to check that the proof of \cite[Proposition 8.6]{CH20} applies with the above definition.
\end{proof}

\begin{remark}
 Proposition \ref{arrangements} holds for any hyperplane arrangement, since Definition \ref{defi : good arrangement} is automatically satisfied in codimension 1. In that case, Proposition \ref{arrangements} reduces to the main result of Kim's paper \cite{Kim94}. In the higher codimension case, the above proposition follows from the more general result proved in \cite[Example 1.14]{delignealgebre}.
\end{remark}

\begin{example}
Consider the space of ordered $n$-configuration points in $\Co^d$ denoted $\mathrm{Conf}_n\left( \Co^d \right)$. Let $(i,j)$ be an unordered pair of distinct elements in $\{1, \cdots, n\}$, and consider the diagonal
\[\Delta_{i,j} = \lbrace (x_1, \dots, x_{n}) \in ( \Co^d )^n, x_i = x_j\rbrace \ .\]
The collection $\{\Delta_{i,j}\}_{(i,j)}$ of codimension $d$ subspaces of $\left( \Co^d \right)^n$ is easily seen to be a good arrangement and the complement $$ ( \Co^d )^n - \cup_{(i,j)} \Delta_{i,j}$$ is exactly $\mathrm{Conf}_n\left( \Co^d \right)$. Using the previous proposition, one can conclude that this space is $2d/(2d -1)\text{-}$pure. 
\end{example}

\subsection{Formality of the singular chains functor} Let $\mathrm{Var}_{\Co}^{\alpha\text{-}pure}$ denotes the full subcategory of $\mathrm{Var}_{\Co}$ whose objects are varieties that are $\alpha\text{-}$pure.

\begin{theorem}[{\cite[Theorem~7.3]{CH20}}]\label{purity implies formality}
Let $\alpha$ be a non-zero rational number. The singular chains functor
\[C_*( -; \QQ) : \mathrm{Var}_{\Co}^{\alpha\text{-}pure} \to \ch_*(\QQ)\]
is formal as lax symmetric monoidal functor.
\end{theorem}

\begin{proof}[Ideas of the proof]
By Corollary \ref{infinity functors}, it suffices to prove that this functor is formal as an $\infty$-lax symmetric monoidal functor. The main ingredient of the proof is that there exists a formal functor
\[\mathcal{D}_*(-)_{\QQ} : \N(\mathrm{Var}_{\Co})^{\times} \to \mathbf{Ch}_*(\QQ)^{\otimes}\]
which is weakly equivalent to $C_*( -; \QQ)$ in the category of strong symmetric monoidal $\infty$-functors. The construction of this functor involves the notion of mixed Hodge complexes introduced by Deligne in \cite{Deligne74}. 

\begin{definition}[Mixed Hodge complex]
 A \emph{mixed Hodge complex} over $\mathbb{Q}$ is the data of
 \begin{itemize}
 	\item[$\centerdot$]  a filtered chain complex $(K_{\mathbb{Q}}, W)$ over $\mathbb{Q}$;
 	\item[$\centerdot$] a bifiltered chain complex $(K_{\Co}, W, F)$ over $\Co$;
 	\item[$\centerdot$] a finite string of filtered quasi-isomorphisms of filtered complexes of $\Co$-vector spaces \[\quad \quad (K_{\mathbb{Q}}, W) \otimes \Co \overset{\alpha_1}{\longrightarrow} (K_1, W) \overset{\alpha_2}{\longleftarrow} \cdots \overset{\alpha_{l-1}}{\longrightarrow} (K_{l-1}, W)  \overset{\alpha_{l}}{\longrightarrow} (K_{\Co}, W).\]
 \end{itemize}
We call $l$ the \emph{length} of the mixed Hodge complex. The following axioms must furthermore be satisfied: 
\begin{itemize}
	\item[($\mathrm{MH}_0$)] The homology $H_*(K_{\mathbb{Q}})$ is bounded and finite-dimensional. \smallskip
	\item[($\mathrm{MH}_1$)] The differential of $Gr^p_W K_{\Co}$ is strictly compatible with $F$. \smallskip
	\item[($\mathrm{MH}_2$)] The filtration on $H_n\left(Gr^p_W K_{\Co}\right)$ induced by $F$ makes $H_n\left(Gr^p_W K_{\mathbb{Q}}\right)$ into a pure Hodge structure of weight $p +n$. 
\end{itemize}
Morphisms of mixed Hodge complexes are given by levelwise bifiltered morphisms of complexes making the corresponding diagrams commute. We denote by $\mathrm{MHC}_{\mathbb{Q}}$ the category of mixed Hodge complexes of a certain fixed length, omitted in the notation. We can view it as a symmetric monoidal category, with the filtered variant of the Künneth formula, since the tensor product of mixed Hodge complexes is again a mixed Hodge complex. 
\end{definition}

Denote by $\mathbf{MHC}_{\mathbb{Q}}$ the $\infty$-category obtained by inverting weak equivalences of mixed Hodge complexes. It can be equipped with a structure of symmetric monoidal $\infty$-category (see \cite{Drew15}). Beilinson gave an equivalence of categories between the derived category of mixed Hodge structures and the homotopy category of shifted version of mixed Hodge complexes. This equivalence can be lifted at the level of symmetric monoidal $\infty$-categories under the form of the following theorem originally due to Drew (see \cite{Drew15}).

\begin{theorem}[{\cite[Theorem~5.4.]{CH20}}]\label{equivalence sombre}
There exists an equivalence of symmetric monoidal $\infty$-categories $\mathbf{Ch}_*(\mathrm{MHS}_{\mathbb{Q}})^{\otimes} \to \mathbf{MHC}_{\mathbb{Q}}^{\otimes}.$
\end{theorem}  

The functor $\mathcal{D}_*(-)_{\QQ}$ is then obtained as the pre-composition of the forgetful functor $\mathbf{MHC}_{\mathbf{\QQ}}^{\otimes} \to \mathbf{Ch}_*(\QQ)^{\otimes}$ by a symmetric monoidal functor \[\mathcal{D}_* : \N(\mathrm{Var}_{\Co})^{\times} \to \mathbf{MHC}_{\QQ}^{\otimes} \ .\] In the case of smooth varieties, it suffices to take a functorial mixed Hodge complex model for the cochains as constructed for instance in \cite{NavarroAznar87} and dualize it, see \cite[Section 6.1]{CH20}. Once one has constructed this functor for smooth varieties, it can be extended to more general varieties by standard descent arguments, see \cite[Theorem 6.7]{CH20}. Finally, one can show that the functor $\mathcal{D}_*(-)_{\QQ}$ is formal as follows. Let $\mathcal{E}$ be a strong monoidal inverse of the equivalence of Theorem \ref{equivalence sombre}. Consider the composite \[ \mathcal{E} \circ \mathcal{D}_* : \N(\mathrm{Var}_{\Co})^{\times} \to  \mathbf{Ch}_*(\mathrm{MHS}_{\mathbf{\QQ}})^{\otimes}.\]
Using Theorem \ref{equivalence sombre}, the functor $\mathcal{D}_*(-)_{\QQ}$ is weakly equivalent to $U \circ \mathcal{E} \circ \mathcal{D}_* $ where $U$ is the forgetful functor \[U : \mathrm{MHS}_{\mathbf{\QQ}} \to \mathrm{Vect}_{\mathbf{\QQ}} \ .\] The restriction of $\mathcal{E} \circ \mathcal{D}_*$ to $\mathrm{Var}_{\Co}^{\alpha \mbox{-}pure}$ lands in $\mathbf{Ch}_*(\mathrm{MHS}_{\QQ})^{\alpha \mbox{-}pure}$, the full subcategory of $\mathbf{Ch}_*(\mathrm{MHS}_{\QQ})$ spanned by chain complexes whose homology is $\alpha$-pure. However, Theorem \ref{Splitting sur Q} and Theorem \ref{forget général} imply that the restriction of the functor
\[U : \mathbf{Ch}_*(\mathrm{MHS}_{\QQ})^{\otimes} \to \mathbf{Ch}_*(\QQ)^{\otimes}\] 
to $\mathbf{Ch}_*(\mathrm{MHS}_{\QQ})^{\alpha \mbox{-}pure} $ is formal. Proposition \ref{composition} thus implies that $U \circ \mathcal{E} \circ \mathcal{D}_* $ is formal, so does $\mathcal{D}_*(-)_{\QQ}$.  
\end{proof}

\begin{example}[Noncommutative little disks operad]
There is a topological model of the noncommutative little disks operad and non-commutative framed little disks operad introduced in \cite{DotsenkoShadrinVallette15}. The operads at stake are two non-symmetric topological operads \[\mathcal{A}s_{\mathbb{S}^1} \quad \mbox{and} \quad \mathcal{A}s_{\mathbb{S}^1} \rtimes \mathbb{S}^1 \] that are given in each arity by a product of copies of $\Co^*$. Using Example \ref{examples of weight computations} and K\"unneth formula, one shows that such a product is $2$-pure. Theorem \ref{purity implies formality} implies that the operads $C_*(\mathcal{A}s_{\mathbb{S}^1}; \QQ)$ and $C_*(\mathcal{A}s_{\mathbb{S}^1} \rtimes \mathbb{S}^1; \QQ)$ are formal.  
\end{example}

\begin{example}[Cazanave's monoid]
We denote by $F_d$ the algebraic variety of degree $d$ algebraic maps from $\Co \mathbb{P}^1$ to itself that send the point $\infty$ to $1$. An element in this variety can be seen as a pair $(f,g)$ of degree $d$ monic polynomials without any common roots. One can show that the weight filtration on $H^*(F_d; \QQ)$ is $2$-pure, see \cite[Proposition~7.6.]{CH20}. In \cite[Proposition~3.1.]{Cazanave12}, Cazanave shows that the variety $\sqcup_d F_d$ has the structure of a graded monoid which is algebraic and compatible in a homotopical sense with the loop space structure on $\mathrm{Map}_*(\mathbb{S}^2, \mathbb{S}^2)$. This implies that the graded monoid in chain complexes $\bigoplus_{d} C_*(F_d; \QQ)$ is formal. 
\end{example}

\begin{example}[Vaintrob]
There is an operad in log schemes whose complex points recover the (framed) little disks operad, see \cite{vaintrobmoduli}. Log schemes have a Hodge structure on their cohomology and we can lift it at the level of chains. Using this, we can deduce the formality of the chain complex associated to (framed) little disks operad. This approach is developed in \cite{vaintrobformality}. 
\end{example}

\subsection{Formality of Sullivan's polynomial forms functor}

The following is inspired by \cite[Section~8.]{CH20}. Recall the functor of polynomial forms \[\Omega^*_{PL}:\mathrm{sSet}\to \mathrm{CDGA}^{op}\] introduced in Section \ref{Key}. Using this functor, one can obtain a contravariant version of Theorem \ref{purity implies formality} as follows.

\begin{theorem}[{\cite[Theorem~8.1.]{CH20}}]\label{sullivan}
Let $\alpha$ be a non-zero rational number. Sullivan's polynomial forms functor  $$\Omega^*_{PL}: (\mathrm{Var}_{\Co})^{op} \to \ch_*(\QQ)$$ is formal as lax symmetric monoidal functor when restricted to varieties whose weight filtration in cohomology is $\alpha \text{-}$pure.  
\end{theorem}

\begin{proof}
The proof is almost the same as the proof of Theorem \ref{purity implies formality}, using a contravariant version of the functor $\mathcal{D}_*$. 
\end{proof}

\begin{example}
	Theorem \ref{sullivan} and Proposition \ref{arrangements} imply that complements of good codimension $d$ arrangements are formal over $\QQ$ (in the sense of Definition \ref{definition : formality of spaces}). The Deligne-Griffiths-Morgan-Sullivan theorem on formality of compact Kähler manifolds is strongly related to the case $\alpha=1$ of the previous theorem. Indeed compact K\"ahler manifolds include smooth and projective algebraic varieties to which the above theorem applies.
\end{example}

\subsection{Formality of Hopf cooperads}
Given a topological group $G$, the graded vector space $H^*(G; \QQ)$ is a Hopf algebra in which the multiplication comes from the diagonal of $G$ and the comultiplication comes from the multiplication of $G$. With this structure on cohomology, one may be interested in the formality of $C^*(G; \QQ)$ as a Hopf algebra. One has to deal with the issue that the multiplication in the Hopf algebra structure at the cochains level is not strictly commutative. On the other hand, if we consider $\Omega^*_{PL}(G)$, the multiplication is strictly commutative but the comultiplication is only coassociative up to homotopy. A similar problem also arises with operads in spaces, the cohomology of an operad in spaces is a Hopf cooperad (see definition below) but the formality as a Hopf cooperad is not so easy to define since the structure is not strict at the cochain level. The purpose of this subsection is to give a framework for studying this question.

\begin{definition}[Hopf cooperad]\label{defi : Hopf cooperad}
A \emph{Hopf cooperad} over a field $\mathbf{k}$ is an operad in the symmetric monoidal category \[(\mathrm{CDGA}_{\mathbf{k}}^{op},\otimes) \ .\]   If we unravel this definition, a Hopf cooperad is a collection of $\mathrm{CDGA}$s indexed by the positive integers $\{A(n)\}_{n\in\mathbb{N}}$ together with 
\begin{enumerate}
	\item a map $A(1)\to\mathbf{k}$,
	\item a symmetric group action of $\mathbb{S}_n$ on $A(n)$ for each $n$,
	\item maps of CDGAs
	\[\circ_i:A(m+n-1)\to A(m)\otimes A(n)\]
	defined for each integer $i\in\{1,\cdots,m\}$.
\end{enumerate}
satisfying the dual axioms of those of an operad. 
\end{definition}

We can make sense of this definition for more general algebraic structures. To do so, we will introduce the language of algebraic theories.

\begin{definition}
An \emph{algebraic theory} is a small category $T$ with finite products. For $\C$ a category with finite products, a \emph{$T$-algebra} in $\C$ is a finite product preserving functor $T \to \C$. The category of $T$-algebras is the category whose objects are $T$-algebras and whose morphisms are natural transformation of functors.
\end{definition}

\begin{example} \leavevmode
\begin{enumerate}
\item Let $\mathrm{FFGrp}$ be the full sub-category of $\mathrm{Grp}$ spanned by free groups on a finite set of generators. Then $\mathrm{FFGrp}^{op}$ is an algebraic theory. It is an instructive exercise to check that the category of $T$-algebras is equivalent to the category of groups. One side of this equivalence is given by the functor
\[\mathrm{Grp}\to\mathrm{Alg}_T\]
sending $G$ to the functor $F\mapsto \mathrm{Hom}(F,G)$.
\item Similarly, there exist algebraic theories for which the $T$-algebras are monoids, abelian groups, rings, operads, cyclic operads, modular operads etc. They formally look very similar to the previous example. One simply takes the opposite of the category of free objects on finitely many generators.
\end{enumerate}
\end{example}

\begin{definition}[Hopf $T$-coalgebras]
Let $T$ be an algebraic theory and $\mathbf{k}$ be a field. Then the \emph{category of dg Hopf $T$-coalgebras} over $\mathbf{k}$ is the opposite of the category of finite product preserving functors from $T$ to the category $\mathrm{CDGA}^{op}$.
\end{definition}

\begin{remark}
Since the cartesian product in $\mathrm{CDGA}^{op}$ is the coproduct in $\mathrm{CDGA}$ and is simply given by the tensor product, it is clear that the above definition generalizes Definition \ref{defi : Hopf cooperad}.
\end{remark}

\begin{definition}[Weak Hopf $T$-coalgebras]
Let $T$ be an algebraic theory and $\C$ a category with products and with a notion of weak equivalences (e.g. a model category). A \emph{weak T-algebra} in $\C$ is a functor $F : T \to \C$ such that for each pair $(s,t)$ of objects of $T$, the canonical map \[F(t \times s) \to F(t) \times F(s)\] is a weak equivalence. In particular, if $\C$ is the category $\mathrm{CDGA}^{op}$, we call these objects \emph{weak dg Hopf $T$-coalgebras.}
\end{definition}

\begin{remark}\label{rigid}
There are rigidification results due to Bernard Badzioch and Julie Bergner (see \cite{Bad02,Ber06}) that imply that for algebraic theory $T$, we have an equivalence of homotopy categories
\[\mathrm{Ho}(\mbox{weak \emph{T}-algebras in }\mathrm{sSet}) \cong \mathrm{Ho}(\mbox{\emph{T}-algebras in }\mathrm{sSet}).\]
This is for example true for group, monoids, operads, cyclic operads.
\end{remark}

\begin{example}
If $X : T \to \mathrm{sSet}$ is a $T$-algebra (or even a weak $T$-algebra), then $\Omega_{PL}^*(X)$ is a weak $T$-algebra in $\mathrm{CDGA}^{op}$. 
\end{example}

\begin{theorem}[{\cite[Theorem~8.18]{CH20}}]
Let $T$ be an algebraic theory and let \[X : T \to \mathrm{Var}_{\Co}\] be a $T$-algebra such that for all $t \in T$, the weight filtration on the cohomology of $X(t)$ is $\alpha \text{-}$pure, for $\alpha \in \QQ^{\times}$. The weak dg Hopf T-coalgebra $\Omega^*_{PL}(X)$ is formal. 
\end{theorem}

\begin{proof}
The result is an immediate consequence of Theorem \ref{sullivan} since being a weak $T$-coalgebra is a property of a functor $T^{op} \to \mathrm{CDGA}$ that is invariant under quasi-isomorphism. 
\end{proof}

\begin{remark}
The fact that $\Omega^*_{PL}(X)$ is formal as a weak dg Hopf $T$-coalgebra implies that the rational homotopy type of $X$ is determined by $H^*(X; \QQ)$ as a $T$-algebra in graded commutative algebras. Indeed, if we apply the derived Sullivan spatial realization functor
\[\langle - \rangle : \mathrm{CDGA}^{op} \to \mathrm{Top},\]
to a weak dg Hopf $T$-coalgebra, we are going to obtain a weak $T$-algebra in rational spaces. If $X$ is a $T$-algebra in spaces, we get a rational model for $X$ in the sense that the map \[X \longrightarrow \langle \Omega^*_{PL}(X) \rangle \] is a rational weak equivalence of weak $T$-algebras whose target is objectwise rational. Thanks to the rigidification results that hold in the case of spaces, see Remark \ref{rigid}, the weak $T$-algebra $\langle \Omega^*_{PL}(X) \rangle$ can be strictified to a strict $T$-algebra. If $X$ is formal, one also get a rational model for $X$ through \[\langle H^*(X; \QQ) \rangle \ .\]   
\end{remark}

\begin{remark}
Let us cite some related work that predates \cite{CH20}. First, Morgan constructed in \cite{morganalgebraic} an explicit small Sullivan model of smooth algebraic varieties which is equipped with a mixed Hodge structure. This model was used by Dupont in order to show that $1$ and $2$-pure smooth algebraic varieties are formal, see \cite{Dup16}. An alternative argument in the $2$-pure case due to Beilinson is explained in \cite[Proposition 3.4]{DH18}. Similar ideas are used by Chataur and Cirici in \cite{CC17} in order to prove the formality of some singular projective algebraic varieties. 
\end{remark}

\section{\textcolor{bordeau}{Galois group actions}}\label{torsion}

So far, we only considered formality problems with coefficients in a field of characteristic zero. By exploiting Galois group actions on étale cohomology rather than mixed Hodge structures, one can derive formality results with torsion coefficient. The formality results obtained in this case is however only up to a certain degree, which depends on the cardinality of the field of coefficients. In this section, we expose these results, based on \cite{CH22}. 

Let $\mathcal{A}$ be a symmetric monoidal abelian category. 

\begin{definition}
	Let $N$ be an integer. A morphism of chain complexes \[f : A \longrightarrow B \in \ch_*(\mathcal{A})\] is called a \emph{$N$-quasi-isomorphism} if the induced morphism in homology $H_i(f) : H_i(A) \to H_i(B)$ is an isomorphism for all $i \leqslant  N$.  
\end{definition}

\begin{definition}
Let $(\mathcal{C}, \otimes, \mathbf{1})$ be a symmetric monoidal category and let \[F : \mathcal{C} \to \ch_*(\mathcal{A})\] be a lax symmetric monoidal functor. The functor $F$ is said to be a \emph{$N$-formal lax symmetric monoidal} functor if there is a string of natural transformations of lax symmetric monoidal functors  $$ F \overset{\Phi_1}{\longleftarrow} F_1 \longrightarrow \cdots \longleftarrow F_n \overset{\Phi_n}{\longrightarrow} H \circ F$$ such that for all $X$ in $\mathcal{C}$, the morphisms $\Phi_i(X)$ are $N$-quasi-isomorphisms. 
\end{definition}

\begin{remark}
We can also extend this definition to a notion of $N$-formal lax symmetric monoidal $\infty$-functor for functors $\N(\C) \to \mathbf{Ch}_*(\mathcal{A})$ as in Definition \ref{formal infinity functors}. 
\end{remark}

\subsection{Some words on étale cohomology}

Our main tool to prove formality results with torsion coefficients uses a Galois group action on étale cohomology. This section recalls some basic notions around this topic. \bigskip

Étale cohomology is a particular example of sheaf cohomology. Recall that if $X$ is a Hausdorﬀ, paracompact and locally contractible topological space (this is the case if $X$ is a smooth manifold), the singular cohomology of $X$ can be computed as the sheaf cohomology with values in the constant sheaf. Let $A$ be an abelian group and denote $\underline{A}$ the constant sheaf. Then there is an isomorphism between singular cohomology and sheaf cohomology:  \[H^*_{sing}(X;A) \cong H^*_{sheaf}(X;\underline{A}) \ . \] 
In fact, this statement can be lifted at the level of cochains. If $R$ is a commutative ring, the dg algebra of singular cochains $C^*(X; R)$ is quasi-isomorphic to the sheaf-cohomology complex $\mathrm{R} \Gamma (X; \underline{R})$. This quasi-isomorphism can be chosen to be compatible with the $E_{\infty}$-algebra structures on both sides, see \cite{petersenremark} or \cite{CC22}. If $X$ is a scheme over some base field $K$, an étale cover of $X$ is a set \[\{p_i : U_i \to X\}\] of étale morphisms locally of finite type which are jointly surjective. The notion of étale morphism can be viewed as an algebro-geometric analogue of the notion of local homeomorphism in topology. Then, étale cohomology of $X$ with coefficients in an abelian group $A$ can be defined as the derived global sections of the constant sheaf with value $A$ on the \'etale site of $X$,
\[H^*_{et}(X;A) \coloneqq H^*\left(\mathrm{R}\Gamma\left(X_{et}; \underline{A}\right)\right).\]

For smooth schemes over $\Co$, étale cohomology is related to the classical cohomology thanks to the following theorem. 

\begin{theorem}[Artin]
 If $X$ is a smooth scheme over $\Co$ and $A$ is a finite abelian group, then there is an isomorphism \[H^*_{et}(X;\underline{A}) \cong H^*_{sheaf}(X_{\mathrm{an}};\underline{A})\] where $X_{\mathrm{an}}$ denotes the complex manifold underlying $X$. 
\end{theorem}

Let $p$ be a prime number and $K$ be a $p$-adic field. We denote by $k$ the residue field of $K$ which is a finite field of characteristic $p$. Let $\iota : \overline{K} \hookrightarrow \Co$ be an embedding. We denote by $\mathrm{Sch}_K$ the category of schemes over $K$ that are separated and of finite type. Let $X$ be a smooth scheme over $K$. We denote by $X_{\mathrm{an}}$ the complex analytic space underlying  $X \times_{K} \Co$. Then we can relate the étale cohomology of $X \times_{K} \overline{K}$ with the one of $X \times_{K} \Co$ using the embedding $\iota$. We obtain a zig-zag of maps 
\[H^*_{et}(X \times_{K} \overline{K};A) \overset{u}{\longleftarrow} H^*_{et}(X \times_{K} \Co;A)\xrightarrow{\cong} H^*_{sheaf}(X_{\mathrm{an}},\underline{A}),\]
for $A$ any finite abelian group. Following a standard theorem of étale cohomology known as \emph{smooth base change theorem} (see \cite{SGA4}), the map $u$ is also an isomorphism. By functoriality of étale cohomology, the group \[H^*_{et}(X \times_{K} \overline{K};A)\] has an action of the absolute Galois group $\mathrm{Gal}(\overline{K}/K)$ . It can be shown that all these isomorphisms lift as $E_{\infty}$-quasi-isomorphisms at the cochain level (see \cite{shin}) and similarly that the Galois action on the left side of this zig-zag lifts to the cochain level. The group $\mathrm{Gal}(\overline{k}/k)$ is isomorphic to the profinite completion of the integers, denoted $\hat{\mathbb{Z}}$, generated by the Frobenius $x \mapsto x^q$, with $q = |k| = p^n$. We make once and for all a choice of a lift $\varphi$ of the Frobenius in $\mathrm{Gal}(\overline{K}/K)$. The upshot of all this discussion is that, given a finite ring $A$, we have a zig-zag
\[C_{et}^*(X \times_{K} \overline{K}, A )  \longleftarrow C_{et}^*(X \times_{K} \Co;A) \longrightarrow C_{sing}^*(X_{\mathrm{an}};A).\]
in which both maps are quasi-isomorphisms of $E_{\infty}$-algebras and in which the left-hand side is equipped with an automorphism $\varphi$. Moreover, this data is functorial in the input $X$. This discussion can be extended from finite coefficients to $\ell$-adic coefficients via the following definition.

\begin{definition}
Let $X$ be a scheme over some base field $K$. Define 
\[H^*_{et}(X; \mathbb{Z}_{\ell}) \coloneqq \lim_n H^*_{et}(X; \mathbb{Z}/\ell^n) \quad \mbox{and}\quad H^*_{et}(X; \mathbb{Q}_{\ell}) \coloneqq H^*_{et}(X; \mathbb{Z}_{\ell}) \otimes_{\mathbb{Z}_{\ell}} \mathbb{Q}_{\ell}.\]
\end{definition}

From the fact that smooth schemes have finite type cohomology, one can also show that for $X$ a smooth scheme over $\Co$, there is an isomorphism \[H^*_{et}(X; \mathbb{Z}_{\ell}) =  H^*_{sheaf}(X_{\mathrm{an}}; \mathbb{Z}_{\ell})\] and in fact all the previous discussion can be applied to the case of $\mathbb{Z}_{\ell}$ and $\mathbb{Q}_{\ell}$-coefficients as well.

\subsection{Formality using étale cohomology}

As in the previous section, $K$ denotes a $p$-adic field and $k = \mathbb{F}_q$ its residue field. Denote by $h$ the order of $q$ in $\mathbb{F}^{\times}_{\ell}$. For $\ell$ some prime number which is prime to $q$. 

\begin{definition}[$q$-Tate modules]
Let $V$ be a finite dimensional $\mathbb{F}_{\ell}$-vector space and $\varphi$ an automorphism of $V$. We say that the pair $(V,\varphi)$ is a \emph{$q$-Tate module} if the eigenvalues of $\varphi$ in $\overline{\mathbb{F}_{\ell}}$ are powers of $q$. Let $n \in \mathbb{N}$. A $q$-Tate module is said to be \emph{pure of weight $n$} if the only eigenvalue of $\varphi$ is $q^n$.  
\end{definition}

\begin{remark}
It should be noted that the weight of a pure Tate module is only well-defined modulo $h$. Observe also that the weights have been divided by $2$ compared to the Mixed Hodge case. 
\end{remark}

\begin{remark}
It can be shown that the category of Tate modules denoted $\mathrm{TMod}$ is a symmetric monoidal abelian category. The kernels and cokernels are simply the kernels and cokernels of the underlying $\mathbb{F}_{\ell}$-vector spaces equipped with the induced action of the endomorphism $\varphi$.  
\end{remark}

\begin{definition}
Let $\alpha$ be a rational number satisfying $0<\alpha<h$. Let $X \in \mathrm{Sch}_K$. We say that the \'etale cohomology of $X$ is \emph{$\alpha$-pure} if the following conditions are satisfied
\begin{enumerate}
\item If $\alpha n\notin \mathbb{Z}$, then $H^n_{et}(X \times_{K} \overline{K}; \mathbb{F}_{\ell}) = 0$. 
\item If $\alpha n\in\mathbb{Z}$, then $q^{\alpha n}$ is the only eigenvalue of the Frobenius acting on \[H^n_{et}(X \times_{K} \overline{K}; \mathbb{F}_{\ell}) \ .\]
\end{enumerate}   
\end{definition}

\begin{example}
Let $X=\mathbb{P}^n$. Then the \'etale cohomology of $X$ is $1/2$-pure
\end{example}

\begin{definition}
A colored operad $\P$ is \emph{admissible} if the category \[\mathrm{Alg}_{\P}(\ch_*(K))\] admits a model structure transferred along the forgetful functor $$\mathrm{Alg}_{\P}(\ch_*(K)) \to \ch_*(K)^{Ob(\P)}.$$ We say that $\P$ is \emph{$\Sigma$-cofibrant} if for all $n \geqslant  1$, and by all surjections $i : [n] \to I$, the symmetric group $\mathbb{S}_n$ acts freely on $\P(n,i)$. A colored operad $\P$ in sets is called \emph{homotopically sound} if it is admissible and $\Sigma$-cofibrant. 
\end{definition}

\begin{theorem}[{\cite[Theorem~6.5]{CH22}}]\label{frobenius}
Let $\P$ be a homotopically sound operad and let $X$ be a $\P$-algebra in $\mathrm{Sch}_K$ such that for each color $c = i(k)$ of $\P$, \[H^*_{et}(X(c) \times_{K} \overline{K}; \mathbb{F}_{\ell} )\] is $\alpha$-pure. The dg $\P$-algebra $C_*(X_{\mathrm{an}}; \mathbb{F}_{\ell})$ is $\lfloor (h-1)/ \alpha \rfloor$-formal. 
\end{theorem}

Let us explain where the homotopically sound hypothesis comes from. Recall that if $\mathbf{k}$ is a field of characteristic zero and if $F$ is formal (resp. $N$-formal) as lax symmetric monoidal $\infty$-functor $\N(\C) \to \mathbf{Ch}_*(\mathbf{k})$, then $F$ is formal (resp. $N$-formal) as a lax symmetric monoidal functor (as in the Corollary \ref{infinity functors}). However for a field which is not of characteristic zero, this corollary fails. This comes from the fact that the homotopy theory of lax monoidal functors is in general not equivalent to the homotopy theory of lax monoidal $\infty$-functors. In order to circumvent this difficulty, we will restrict to homotopically sound operads for which we have a rigidification result due to Hinich. For $\P$ an operad in sets, we denote by \[\mathbf{Alg}_{\P}(\mathbf{Ch}_*(\mathbf{k}))\] the $\infty$-category of $\P$-algebras in the $\infty$-category of chain complexes of $\mathbf{k}$-vector spaces. There is an obvious functor
\[\mathrm{Alg}_{\P}(\ch_*(\mathbf{k}))\to\mathbf{Alg}_{\P}(\mathbf{Ch}_*(\mathbf{k}))\]
which sends quasi-isomorphisms to equivalences. It  induces a map
\[\mathrm{N}_{W}(\mathrm{Alg}_{\P}(\ch_*(\mathbf{k})))\to\mathbf{Alg}_{\P}(\mathbf{Ch}_*(\mathbf{k}))\]
Hinich shows that, under the hypothesis that $\P$ is homotopically sound, this functor is an equivalence of $\infty$-categories. In particular, we obtain the following proposition as a corollary of Hinich's theorem.

\begin{proposition}\label{hinich torsion}
Let $\P$ be a homotopically sound operad in sets. Let $A$ be a $\P$-algebra in $\ch_*(\mathbf{k})$. 
\begin{enumerate}
\item If $A$ is formal in  $\mathbf{Alg}_{\P}(\mathbf{Ch}_*(\mathbf{k}))$, then $A$ is formal in $\mathrm{Alg}_{\P}(\ch_*(\mathbf{k}))$. \smallskip
\item If $A$ is $N$-formal in $\mathbf{Alg}_{\P}(\mathbf{Ch}_{\geqslant  0}(\mathbf{k}))$, then $A$ is $N$-formal in $\mathrm{Alg}_{\P}(\ch_{\geqslant  0}(\mathbf{k}))$. 
\end{enumerate}
\end{proposition}

\begin{proof}
The point (1) is an immediate consequence of Hinich's theorem. Point (2) follows from (1) and from the observation that a $\P$-algebra $A$ is $N$-formal if and only if its truncation $t_{\leqslant  N}(A)$ is formal (where $t_{\leqslant  N}$ denotes the truncation functor which kills homology in degrees greater than $N$). 
\end{proof}

Let $h$ be a positive integer. If $\mathcal{A}$ is a symmetric monoidal abelian category, we denote by $gr^{(h)} \mathcal{A}$ the category of $\mathbb{Z}/h$-graded objects of $\mathcal{A}$. An object in this category is a collection \[\{A^a\}_{a\in\mathbb{Z}/h}\] of objects of $\mathcal{A}$ indexed by the elements of the group $\mathbb{Z}/h$. It is a symmetric monoidal category, with the tensor product given by \[(A \otimes B)^n \coloneqq \sum_{a + b \equiv n \pmod h} A^a \otimes B^b. \]

There is a functor $\mathrm{Tot} : gr^{(h)} \mathcal{A} \to \mathcal{A}$ given by the formula
\[\{A^a\}_{a\in\mathbb{Z}/h}\longmapsto \bigoplus_{a\in\mathbb{Z}/h}A^a\]
It is straightforward that this functor can be given the structure of a strong symmetric monoidal functor. We have the following version of Theorem \ref{Splitting sur Q}. 

\begin{lemma}[{\cite[Lemma~2.9]{CH22}}]\label{splitting torsion}
Let $h$ be the order of $q$ in $(\mathbb{F}_{\ell})^{\times}$. The  functor $U : \mathrm{TMod} \to \mathrm{Vect}_{\mathbb{F}_{\ell}}$ defined by $(V, \varphi) \mapsto V$ admits a factorization

\begin{center}
	\begin{tikzcd}
		&gr^{(h)} \mathrm{Vect}_{\mathbb{F}_{\ell}} \arrow[d,"\mathrm{Tot}"]\\
		\mathrm{TMod} \arrow[r,"U"'] \arrow[ru, "G"] & \mathrm{Vect}_{\mathbb{F}_{\ell}}
	\end{tikzcd}
\end{center}

into strong symmetric monoidal functors. The functor $G$ is defined by declaring $G(V, \varphi)^n$ to be the generalized eigenspace for the eigenvalue $q^n$. 
\end{lemma}

To study formality in the torsion case, there is also a version of Theorem \ref{forget général}, which deals with chain complexes of $\mathbb{Z}/h$-graded objects. We denote by 
\[\ch_*(gr^{(h)}\mathcal{A})^{\alpha \mbox{-}pure}\subset \ch_*(gr^{(h)}\mathcal{A})\] the full subcategory given by those $\mathbb{Z}/h$-graded complexes $V = \bigoplus V^p_n$ with \emph{$\alpha\text{-}$pure homology}~:
\[H_n(V)^p = 0 \quad \mbox{for all } p \neq \alpha n \pmod m .\]

The following proposition is a $\mathbb{Z}/h$-graded version of Proposition \ref{forget général}. 

\begin{proposition}[{\cite[Proposition~5.13.]{CH22}}]\label{forget torsion}
The forgetful functor  $$U : \ch_{\geqslant  0}(gr^{(h)}\mathcal{A})^{\alpha \text{-} pure} \to \ch_{\geqslant  0}(\mathcal{A})$$ is $\lfloor (h-1)/ \alpha \rfloor$-formal as lax symmetric monoidal functor.  
\end{proposition}

\begin{definition}[Tate complex]
A \emph{Tate complex} is a pair $(C, \varphi)$ where $C$ is a chain complex of $\mathbb{F}_{\ell}$-modules and $\varphi$ is an endomorphism of $C$ such that the pair $(H_n(C), H_n(\varphi))$ is a Tate module for all $n$. We denote by $\mathbf{TComp}$ the $\infty$-category of Tate complexes and by $\mathbf{TComp}^{\alpha \text{-} pure}_{\geqslant  0}$ to be the full subcategory of $\alpha$-pure non-negatively graded Tate complexes. 
\end{definition}

\begin{proof}[Proof of Theorem \ref{frobenius}]
It suffices to prove that the forgetful functor
\[\mathbf{TComp}^{\alpha \text{-} pure}_{\geqslant  0} \to \mathbf{Ch}_{\geqslant  0}(\mathbb{F}_{\ell})\]
is $N$-formal as a lax monoidal $\infty$-functor, with $N = \lfloor (h-1)/ \alpha \rfloor$. Indeed, if this is true, this will imply that $C_*(X_{\mathrm{an}}; \mathbb{F}_{\ell})$ is $N$-formal in the $\infty$-category $\mathbf{Alg}_{\P}(\mathbf{Ch}_*(\mathbf{k}))$ and Proposition \ref{hinich torsion} allows to come back to the standard definition of formality. To do so, we use the same strategy as in the mixed Hodge complexes case. First, one can prove a Beilinson type theorem for the categories of Tate complexes. There is a canonical symmetric monoidal functor $\ch_*(\mathrm{TMod})  \to \mathrm{TComp}$ that preserves quasi-isomorphisms on both sides. Therefore it induces a symmetric monoidal $\infty$-functor 
\[\mathbf{Ch}_*(\mathrm{TMod}) \to  \mathbf{TComp}.\]
One can show that this $\infty$-functor is in fact an equivalence of symmetric monoidal $\infty$-categories,  see \cite[Theorem~4.7.]{CH22}. Using Lemma \ref{splitting torsion}, the forgetful functor $\mathbf{TComp} \to \mathbf{Ch}_{\geqslant  0}(\mathbb{F}_{\ell})$ factorizes as follows 
$$\mathbf{TComp} \to \mathbf{Ch}_*(\mathrm{TMod}) \to \mathbf{Ch}_*\left(gr^{(h)}\mathrm{Vect}_{\mathbb{F}_{\ell}}\right) \to  \mathbf{Ch}_{*}(\mathbb{F}_{\ell}).$$ where the first map is an inverse to the equivalence mentioned before. We can restrict all categories to $\alpha$-pure $\mathbb{Z}/h$-graded objects and then the last functor is $N$-formal by Proposition \ref{forget torsion}. The result follows from Proposition \ref{composition}. 
\end{proof}

We now give some examples of applications of this theorem.

\begin{proposition}\label{frobenius action on cohomology of moduli space}
Let $K$ be any $p$-adic field. The cohomology $H^*_{et}(\overline{\mathcal{M}_{0,n}}\times_{\mathbb{Z}}\overline{K},\mathbb{F}_{\ell})$ is $1/2$-pure.
\end{proposition}

\begin{proof}
A smooth scheme over $K$ is said to be $1/2$-pure if its \'etale cohomology is $1/2$-pure. We first make the following claim.
\begin{itemize}
\item[$\centerdot$] The set of $1/2$-pure schemes is stable under finite products.
\item[$\centerdot$] If $Z\to X$ is a closed embedding of smooth schemes and $Z$ and $X$ are $1/2$-pure, then the blow-up $B_Z(X)$ is also $1/2$-pure.
\end{itemize}

The first property is an immediate consequence of the K\"unneth formula in \'etale cohomology. The second property follows from the blow-up formula which gives an equivariant isomorphism
\[H_{et}^*(B_Z(X);\mathbb{F}_{\ell})\cong H_{et}^*(X;\mathbb{F}_{\ell})\oplus\left(\bigoplus_{i=1}^{c-1} H_{et}^*(Z;\mathbb{F}_{\ell})[-2i]\otimes_{\mathbb{F}_{\ell}}\mathbb{F}_{\ell}(-i)\right)\]
where $\mathbb{F}_{\ell}(-i)$ denotes the Tate module in which the Frobenius acts by multiplication by $q^i$ and $c$ denoted the codimension of $Z$ in $X$. Now, we can prove the proposition by induction on $n$. For $n=3$, this moduli space is a point. For $n=4$, we have $\overline{\mathcal{M}}_{0,4}\cong\mathbb{P}^1$ and the proposition is a classical computation. Assume that the proposition has been proved for $\{3,4,\cdots, n\}$. We may use Keel's inductive description of $\overline{\mathcal{M}}_{0,n+1}$ as a sequence of blow-ups starting from $\overline{\mathcal{M}}_{0,n}\times\overline{\mathcal{M}}_{0,4}$ and in which, at each stage, the variety that is blown-up is isomorphic to $\overline{\mathcal{M}}_{0,p+1}\times \overline{\mathcal{M}}_{0,q+1}$ with $p+q=n$, see \cite[Section 1]{Keel92}. We conclude by the induction hypothesis and the first claim of the proof.
\end{proof}

\begin{example}
Consider the operad $\mathcal{O}=\{\overline{\mathcal{M}}_{0,*+1}\}$ of moduli spaces of stable algebraic curves of genus $0$ of Example \ref{moduli}. Let us pick a prime number $p$ that is different from $\ell$ and such that $p$ has order $\ell -1$ in $\mathbb{F}_{\ell}^{\times}$ (such a prime exists thanks to Dirichlet's theorem on arithmetic progressions). The hypothesis of Theorem \ref{frobenius} are satisfied with $\alpha=1/2$ and $h=\ell -1$ thanks to Proposition \ref{frobenius action on cohomology of moduli space} above. In this case, the operad $\P$ is the colored operad whose algebras are operads. This operad is homotopically sound. We can therefore conclude that the dg- operad $C_*(\mathcal{O};\mathbb{F}_{\ell})$ is $2(\ell -2)$ formal.
\end{example}

\begin{example}Let $\mathcal{D}_2$ denotes the little disks operad. This is not quite an operad in the category of smooth schemes but in any case, one can construct a model $\mathcal{E}$ of $C_*(\mathcal{D}_2; \mathbb{F}_{\ell})$ equipped with an action of the profinite Grothendieck-Teichmüller group $\widehat{GT}$. This is very similar to what we explained in Example \ref{little disks operad}. The group $\widehat{GT}$ is equipped with a surjective homomorphism
\[\widehat{GT}\to\widehat{\mathbb{Z}}^{\times} \cong\prod_{p}\mathbb{Z}_p^{\times}\]
where $\widehat{\mathbb{Z}}^\times$ is the group of units in the profinite completion of the ring of integers. We may pick an element $\varphi$ of $\widehat{GT}$ that lifts $p\in\mathbb{Z}_{\ell}^{\times}$ where $p$ is some prime number distinct from $\ell$. Then, it can be shown that $\mathcal{E}$ equipped with the automorphism $\varphi$ is an operad in the category of Tate complexes. Moreover, the homology $H_*(\mathcal{D}_2(k);\mathbb{F}_{\ell})$ is $1$-pure for each $k$. Using the method of Theorem \ref{frobenius} we deduce that the operad $\mathcal{E}$, and therefore the operad \[C_*(\mathcal{D}_2(k);\mathbb{F}_{\ell}) \ ,\] is $(\ell -2)$-formal, see \cite[Theorem 6.7.]{CH22}. It should be noted that this result is sharp since it can be shown that the operad $C_*(\mathcal{D}_2;\mathbb{F}_{\ell})$ is not $(\ell -1)$-formal. 	
\end{example}

\section{\textcolor{bordeau}{Homotopy transfer and formality}}\label{7}

Formality aims to measure if the induced structure in homology retains all of the homotopical information contained in a given algebra. Through an operadic approach, one can make this intuition precise and derive another characterization of formality: \emph{gauge formality}. In this section, we present this other approach and formality criteria based on it.

\begin{assumption}\leavevmode
	\begin{itemize}
		\item[$\centerdot$] We suppose that every operad or cooperad in $\mathrm{Ch}_*(R)$ is reduced, connected and has an additional weight grading. 
		
		\item[$\centerdot$] Let $\P$ be an operad in the category of $R$-modules, concentrated in degree zero. We assume that either $R$ is a $\QQ$-algebra or that $\P$ is a non-symmetric operad. 
		
		\item[$\centerdot$] Let $\C$ be a conilpotent cooperad over $R$, with coaugmentation coideal $\overline{\C}$ concentrated in strictly positive degree.  We assume that we are given the datum of a Koszul morphism $\C\to\P$, i.e. a twisting morphism that induces a quasi-isomorphism
		\[\P_{\infty} \coloneqq  \Omega \C\stackrel{\sim}{\longrightarrow} \P \ .\]
	\end{itemize}
\end{assumption}

\subsection{Gauge formality}
There exist several equivalent characterizations of a $\P_{\infty}$-algebra structure. We are going to use the one in terms of coderivations, see \cite[Section~10.1]{LodayVallette12} for more details. 

\begin{proposition}\label{dix-neuf}
	A $\P_{\infty}$-algebra structure on a chain complex $A$ is equivalent to a codifferential of $\C(A)$, i.e. a degree $-1$ square-zero coderivation \[M \in \mathrm{Coder}(\C(A))\]  of the cofree conilpotent coalgebra $\C(A)$. 	An \emph{$\infty$-morphism} between two $\P_{\infty}$-algebras 	\begin{center}
		\begin{tikzcd}[column sep=normal]
			F :  (A,M) \ar[r,squiggly] 
			& (A', M') \ .
		\end{tikzcd} 
	\end{center}
	is a map of dg $\C$-coalgebras $(\C(A), M) \to (\C(A'), M')$.
\end{proposition}

\begin{remark}
	A coderivation $M$ is completely determined by its projection on the cogenerators $m : \C(A) \to A \ .$ In other words, there is an isomorphism
	\[\mathrm{Coder}(\C(A))\cong \Hom(\C(A);A) \ , \] see \cite[Proposition~6.3.8]{LodayVallette12} for more details. These two points of view will be used equally by keeping the uppercase letters for the coderivations and the lowercase letters for the associated projections. Similarly, an $\infty$-morphism $F$ is completely determined by its projection $f : \C(A) \to A'$. 
\end{remark}

For every coderivation $M$, we denote by $m_i$ the restriction to \[\C_i(A) \coloneqq \bigoplus_{k \in \mathbb{N}} \C_i(k) \otimes_{\S_k} A^{\otimes k} \ . \] Similarly, we denote by $f_i$ the restriction of an $\infty$-morphism $F$ to $C_i(A)$. Since $\C_0 = I$, the component $f_0$ is an endomorphism $A \to A$.

\begin{definition}
	An $\infty$-morphism $F : A \rightsquigarrow A'$ is an \emph{$\infty$-quasi-isomorphism} (resp. $\infty$-isomorphism) if its first component $f_0 : A \to A'$ is a quasi-isomorphism (resp. isomorphism). 
\end{definition}

Given a $\P_{\infty}$-algebra $(A,M)$, there is a natural induced $\P$-algebra structure $(H(A),m_*)$ on the homology. In general, this structure forgets a part of the homotopical information: most structures are not formal. However, when there is a homotopy retraction between $A$ and $H(A)$, there is another way to transfer a given structure to the homology without loss of homotopical information. It is given by the homotopy transfer theorem.

\begin{theorem}[Homotopy transfer theorem] 
	Let $A$ be a chain complex and let  \[
	\hbox{
		\begin{tikzpicture}
			
			\def\upshift{0.075}
			\def\downshift{0.075}
			\pgfmathsetmacro{\midshift}{0.005}
			
			\node[left] (x) at (0, 0) {$(A,d)$};
			\node[right=1.5 cm of x] (y) {$(H(A),0)$};
			
			\draw[->] ($(x.east) + (0.1, \upshift)$) -- node[above]{\mbox{\tiny{$p$}}} ($(y.west) + (-0.1, \upshift)$);
			\draw[->] ($(y.west) + (-0.1, -\downshift)$) -- node[below]{\mbox{\tiny{$i$}}} ($(x.east) + (0.1, -\downshift)$);
			
			\draw[->] ($(x.south west) + (0, 0.1)$) to [out=-160,in=160,looseness=5] node[left]{\mbox{\tiny{$h$}}} ($(x.north west) - (0, 0.1)$);
	\end{tikzpicture}} \] be a homotopy retraction where $i$ is a quasi-isomorphism, $ip - \mathrm{id}_A = d_A h + h d_A\ ,$  and $pi = \mathrm{id}_{H(A)} \ .$ Let $(A,M)$ be a $\P_{\infty}$-algebra structure.	
	\begin{enumerate}
		\item There exists a transferred $\P_{\infty}$-algebra structure $(H(A), M^{t})$ such that \[m^t_1 = m_* \ .\]
		\item The inclusion $i$ and the projection $p$ extend to mutually quasi-inverse $\infty$-quasi-isomorphisms, which we will denote by $i_{\infty}$ and $p_{\infty}$.	
		\item The transferred structure is independent of the choice of sections of $H(A)$ on $A$ in up to $\infty$-isomorphisms.
	\end{enumerate}
\end{theorem}

\begin{proof} We refer the reader to \cite{berglund},  \cite[Section~10.3]{LodayVallette12} and references therein. 
\end{proof}

\begin{remark}
Of course, the existence of a homotopy retraction of this form is not automatic if the base ring $R$ is not a field. If the base ring is a principal ideal domain it holds if $H(A)$ is degreewise projective.
\end{remark}

\begin{definition} \label{massey}
	Let $n \in \mathbb{N}^*$ and a dg $\P$-algebra $(A,d,m)$ such that $H(A)$ is a homotopy retract of $A$. The algebra $A$ is said
	\begin{itemize}
		\item[$\centerdot$] \emph{gauge formal} if there exists an $\infty$-quasi-isomorphism	\begin{center}
			\begin{tikzcd}[column sep=normal]
				(H(A),M^t) \ar[r,squiggly,"\sim"] 
				& (H(A), m_*) \ .
			\end{tikzcd} 
		\end{center}
		\item[$\centerdot$] \emph{gauge $n$-formal} if there exist a $\P_{\infty}$-algebra structure $(H(A),R)$ with $r_0 = 0$ and $r_i = 0$ for $i$ such that $2 \leqslant  i \leqslant  n$ and  $\infty$-quasi-isomorphism \begin{center}
			\begin{tikzcd}[column sep=normal]
				(H(A), M^t) \ar[r,squiggly,"\sim"] 
				& (H(A), R) \ .
			\end{tikzcd} 
		\end{center}  
	\end{itemize}
\end{definition}

\begin{remark}\leavevmode
	\begin{enumerate}
		\item	Over a characteristic zero field, gauge formality is equivalent to formality, see \cite[Theorem~11.4.9]{LodayVallette12}. Over a general ring (and under the assumption that the operad $\P$ is nonsymmetric) this is also true under mild flatness hypothesis, see \cite[Proposition 1.15]{DCH21}.
		\item The terminology of gauge formality comes from the equivalence between the existence of $\infty$-quasi-isomorphisms is this situation and the existence of gauge equivalence in a certain dg Lie algebra, see  \cite[Section~2]{Kaledin} for more details. 
	\end{enumerate}
\end{remark}

\subsection{Automorphism lifts}

In \cite{DCH21}, Drummond-Cole and the second author present another proof of Theorem  \ref{purity implies formality} under sightly different assumptions and using gauge formality approach.

\begin{definition}
	Let $V$ be a graded $R$-module. Let $\alpha$ be a unit in $R$. The \emph{degree twisting} by $\alpha$, denoted $\sigma_{\alpha}$, is the linear automorphism of $V$ which acts on the degree $n$ homogenous component of $V$ via multiplication by $\alpha^n$.
\end{definition}

\begin{theorem}[{\cite[Main Theorem.]{DCH21}}]\label{Geoffroy-Gabriel} Let $A$ be a chain complex such that $H(A)$ is a homotopy retract. Let $(A, M)$ be a $\P_{\infty}$-algebra structure. Let $\alpha$ be a unit in $R$ and suppose that the degree twisting $\sigma_{\alpha}$ on $H(A)$ admits a chain level lift, i.e. there exists an $\infty$-quasi-isomorphism $v$ of $(A,M)$ such that $H\left(v_0\right) = \sigma_{\alpha}$.
	\begin{itemize}
		\item If $\alpha^k -1$ is a unit of $R$ for $k \leqslant  n$, then $(A,M)$ is gauge $n$-formal. 
		\item If $\alpha^k -1$ is a unit of $R$ for all $k$, then $(A,M)$ is gauge formal. 
	\end{itemize}
\end{theorem}

\begin{remark}
	This result generalizes to more general types of algebraic structures (colored operads, properads,...) and to other types of homology automorphism, see \cite[Theorem~4.10]{Kaledin}.
\end{remark}

\begin{example}[Complement of subspace arrangements]\label{complement of subspaces arrangements}
	Let $X$ be a complement of hyperplanes arrangement over $\mathbb{C}$, i.e. a complement of a finite collection of affine hyperplanes in $\mathbb{A}^n_{\mathbb{C}}$ viewed as a scheme over $\mathbb{C}$. Let $p$ and $\ell$ be two different prime numbers. Suppose that $X$ can be defined over a finite extension $K$ of $\mathbb{Q}_p$, i.e. there exist an embedding $K \to \Co$ and a complement of a hyperplane arrangement over $K$ denoted $\mathcal{X}$ such that 
	\[X \cong \mathcal{X} \times _{K} \Co \ .\] 
We denote by $q$ cardinality of the residue field of the ring of integers of $K$ and $h$ the order of $q$ in the group of units of $\mathbb{F}_\ell$. We can apply \ref{Geoffroy-Gabriel} to prove that the algebra  $C^*_{\mathrm{sing}}(X_{\mathrm{an}};  \mathbb{Z}_{\ell})$ is $(h-1)$-formal, where $X_{\mathrm{an}}$ denotes the complex analytic space underlying $X_{\mathbb{C}} = X \times _{K} \mathbb{C}$ and where $s$ is the order of $q$ in $\mathbb{F}_{\ell}^{\times}$~. Indeed, as in the previous section, there exists a zig-zag of quasi-isomorphisms of dg associative algebras
	\[C^*_{\mathrm{sing}}(X_{\mathrm{an}};  \mathbb{Z}_{\ell}) \xleftarrow{\sim} C^*_{\mathrm{\acute{e}t}}(\mathcal{X}_{\mathbb{C}};  \mathbb{Z}_{\ell}) \xrightarrow{\sim} C^*_{\mathrm{\acute{e}t}}(\mathcal{X}_{\overline{K}};  \mathbb{Z}_{\ell}) \ .\] One can show that the action of a Frobenius lift on $H^n_{et}(\mathcal{X}_{\overline{K}};  \mathbb{Z}_{\ell})$ is given by multiplication by $q^n$ (see \cite[Theorem 1']{Kim94} for a proof). 
\end{example}
	
\begin{remark}
	The condition of being defined over $K$ is essential here: for each $\ell$ there exists a complement hyperplane arrangement defined over $\Co$ inducing non-trivial Massey products in $H^2( -; \mathbb{F}_{\ell})$, see \cite{mateimassey}. 
\end{remark}

\begin{remark}
The notion of gauge $n$-formality is fundamental in this example. In \cite[Theorem 7.12, (iii)]{CH22} it is wrongly claimed that under these assumptions, $C^*_{sing}(X_{an};\mathbb{F}_\ell)$ is $(h-1)$-formal in the sense that there is a zig-zag of morphisms connecting this dg-algebra to its cohomology and that induce isomorphisms in cohomological degree $\leq h-1$. This statement is incorrect as explained in \cite{CH22corrigendum}. Gauge $n$-formality seems to be the best way to express the kind of partial formality that we have in this situation. On the other hand, similar results hold for complements of subspace arrangements of higher codimension and in this case the two notions of partial formality can be used. 
\end{remark}

\begin{example}[Coformality of configuration spaces]\label{coformality}
	For $d \geqslant  3$, the configuration space $X = \mathrm{Conf}_n(\mathbb{R}^d)$ is coformal, i.e. the dg algebra $C_*(\Omega X; \mathbb{Z}_{\ell})$ is $(\ell-2)(d-2)$-formal, for a given prime number $\ell$, see \cite[Theorem~4.16]{DCH21}. This result is proved using an action of the profinite Grothendieck-Teichm\"uller group on these spaces. Let us mention that the terminology coformal is unusual in this context. Usually a space is called coformal if the Quillen model for its rational homotopy type is formal as a differential graded Lie algebra. This definition does not generalize well when the ring of coefficients is not a field of characteristic zero. However, Saleh has proved that coformality of a based connected topological space $X$ is equivalent to formality of the differential graded algebra of chains on its based loop space $C_*(\Omega X;\mathbb{Q})$ (see \cite{Sal17}). The latter condition makes perfect sense for any ring of coefficients and explains our choice of terminology.
\end{example}

\subsection{Kaledin classes}
Given a dg associative algebra over a characteristic zero field, Kaledin constructs in \cite{Kal07} a class in the associated Hochschild cohomology, which vanishes if and only if the algebra is formal. The work of Kaledin was extended by Melani and Rubió in \cite{MR19} for algebras over a binary Koszul operad in characteristic zero and by the first author in \cite{Kaledin} for algebras over groupoid colored operad or properad, over a commutative ground ring $R$. \bigskip

Let $H$ be a graded $R$-module and let $(H, M)$ be a $\P_{\infty}$-algebra structure. Recall that the complex $ \mathrm{Coder}(\C(H))$ can be equipped with a complete dg Lie algebra structure, whose filtration is defined for all $p \geqslant  0$ by $$\mathcal{F}^p \mathrm{Coder}(\C(A)) \coloneqq \prod_{k \geqslant  p } \Hom \left(\C_p(A), A\right) \ , $$ and whose differential is given by the operator $d_M \coloneqq [M, -]$  see e.g. \cite[Section 6.4]{LodayVallette12}.  Let us consider the \emph{prismatic decomposition} \[\partial_{\hbar} M = m_2 + 2m_3 + 3m_4 + \cdots  \] This element is a cycle in $ \mathfrak{g}^M \coloneqq(\mathcal{F}^1\mathfrak{g}, d_M)$, see e.g. \cite[Proposition~1.14]{Kaledin}.

\begin{definition}[Kaledin class]
	The Kaledin class of $(H,M)$ is the class \[K_M \coloneqq \left[\partial_{\hbar} M\right] \in H_{-1}\left( \mathfrak{g}^M\right) \ . \]
	
	\noindent The \emph{n-truncation} $K_M^{ n}$ of $K_M$ is the class associated to the cycle $$m_2 + 2m_3 + \cdots + (n-1) m_{n}$$ in the cohomology of the complex $\mathcal{F}^1 \mathrm{Coder}(\C(A))/ \mathcal{F}^{n+1}$. 
\end{definition}

\begin{lemma}[Invariance of the Kaledin classes under $\infty$-quasi-isomorphism] \label{invariance par iso} An $\infty$-isomorphism between two $\P_{\infty}$-algebras	$F : (H,M) \rightsquigarrow (H,N)$ induces an isomorphism of dg Lie algebras $F : \mathfrak{g}^M \to \mathfrak{g}^N $ and the following equality holds \[ K_{N} = \left[F\left( \partial_{\hbar} M \right) \right] \ . \] The same goes for the $n$-truncations.
\end{lemma}

\begin{proof}
	We refer the reader to \cite[Lemma~2.38]{Kaledin}.
\end{proof}

\begin{proposition}\label{troncation nulle} Let $(H,M)$ be a $\P_{\infty}$-algebra. Let $n \geqslant 1$ such that $n !$ is a unit in $R$. The following propositions are equivalent. 
	\begin{enumerate}
		\item The $n$-truncation $K_{M}^{n}$ is zero.
		\item There exist a $\P_{\infty}$-algebra structure $(H,N)$ with $n_0 = 0$ and $n_i = 0$ for $i$ such that $2 \leqslant  i \leqslant  n$ and  $\infty$-quasi-isomorphism \begin{center}
			\begin{tikzcd}[column sep=normal]
				(H, M) \ar[r,squiggly,"\sim"] 
				& (H, N) \ .
			\end{tikzcd} 
		\end{center}  
	\end{enumerate}
\end{proposition}

\begin{proof} 
	Let us suppose that $R$ is a characteristic zero field and present the proof of \cite[Proposition~2.9]{MR19}. If $(2)$ holds, the Kaledin class $K_N^{n}$ is zero and so does $K_M^{n}$ by Lemma \ref{invariance par iso}. Let us prove the converse result by induction on $n$. The case $n=1$ is clear. Suppose that the result holds for $n-1$ and that $K_M^{n} = 0$. In particular, we also have $K_M^{n-1} = 0$ and we can assume that \[m_2 = \cdots = m_{n-1} = 0 \ , \] without loss of generality. Thus, we have \[K_M^n = \left[(n-1)m_n\right] = 0\] and there exists $T$ in $\mathfrak{g}$ such that \[d_M(T) \equiv (n-1)m_n \pmod{\mathcal{F}^{n+1}} \ .  \] Considering the coderivation $\tau \coloneqq \frac{t_{n-1}}{n-1}$, we get $[m_1, \tau] = m_n$. One can define an exponential coderivation $E^{\tau}$ by $$e^{\tau} \coloneqq \mathrm{id} + \tau + \frac{\tau^{\circ 2}}{2} + \cdots + \frac{\tau^{\circ k}}{k !}  + \cdots \ .$$ We obtain a $\P_{\infty}$-algebra $(H,N)$ by considering $N = E^{\tau} N E^{-\tau}$. By construction, the element $E^{\tau}$ induces an $\infty$-isomorphism \begin{center}
		\begin{tikzcd}[column sep=normal]
			(H, M) \ar[r,squiggly,"\sim"] 
			& (H, N) \ .
		\end{tikzcd} 
	\end{center}   By construction, we have $n_i = m_i$ for all $i < n$ and $n_n = m_n - [m_2, \tau] = 0$. We refer the reader to \cite[Proposition~2.32]{Kaledin} for a proof over a commutative ring $R$. 
\end{proof}

\begin{theorem}
	Let $(A,M)$ be a $\P_{\infty}$-algebra such that $H(A)$ is a homotopy retract of $A$. If $R$ is a $\mathbb{Q}$-algebra, the algebra $(A,M)$ is gauge formal if and only if the Kaledin class $K_{M^t}$ of a transferred structure $(H(A),M^t)$ is zero. 
\end{theorem}

\begin{proof}
	Let us fix $(H(A),M^t)$ a transferred structure. If $(A,M)$ is gauge formal, there exists an $\infty$-quasi-isomorphism	\begin{center}
		\begin{tikzcd}[column sep=normal]
			F :	(H(A),M^t) \ar[r,squiggly,"\sim"] 
			& (H(A), m_*) \ ,
		\end{tikzcd} 
	\end{center} and $K_{M^t}  = 0$, by Lemma \ref{invariance par iso}. Conversely, suppose that $K_{M^t} = 0$. Let $\iota \geqslant  2$ the smallest integer such that $m_{\iota} \neq 0$. Taking up the demonstration of the previous proposition, there exists a $\P_{\infty}$-algebra structure $(H(A),N)$  where $n_1 = m_1$, $n_2 = \cdots = n_{\iota} = 0$, and an $\infty$-isomorphism \begin{center}
		\begin{tikzcd}[column sep=normal]
			E^{\tau_{\iota}} :	(H(A),M^t) \ar[r,squiggly,"\sim"] 
			& (H(A), N) \ ,
		\end{tikzcd} 
	\end{center} where $\tau_{\iota}$ is the projection of a coderivation of weight $\iota -1$. This procedure can be iterated for any $i \geqslant  \iota$. We obtain a series of $\infty$-isomorphism $E^{\tau_{i}}$. The composition $F = \cdots \circ E^{\tau_{\iota + 1}} \circ E^{\tau_{\iota}} $ is well defined since each $\tau_i$ correspond to a coderivation of weight $i-1$ and leads to the desired $\infty$-isomorphism. 
\end{proof}

As a corollary of this Theorem, we can deduce the following very general result for descent of formality.

\begin{theorem}[Formality descent]\label{descent2}
	Let $S$ be a faithfully flat commutative $R$-algebra. Let $A$ be a chain complex such that $H(A)$ is an $R$-module of finite presentation and a homotopy retract of $A$. Let $(A, M)$ be a $\P_{\infty}$-algebra. Let us denote \[A_{S} \coloneqq A \otimes_{R} S \ .\] 
	\begin{enumerate}
		\item Let $n \geqslant 1 $ be such that $n!$ is a unit in $R$. The algebra $(A, M)$ is gauge $n$-formal if and only if $(A_{S}, M \otimes 1)$ is gauge $n$-formal.
		\item If $R$ is a $\mathbb{Q}$-algebra, the algebra $(A, M)$ is gauge formal if and only if $(A_{S}, M \otimes 1)$ is gauge formal.
	\end{enumerate} 
\end{theorem}

\begin{proof}
	We refer the reader to \cite[Theorem~4.1]{Kaledin}.
\end{proof}

\begin{example}
Combining this result with Example \ref{complement of subspaces arrangements} or Example \ref{coformality}, we deduce that the formality result for complement of hyperplane arrangements can be descended to the localized ring $\mathbb{Z}_{(\ell)}$, see \cite[Theorem~4.2]{Kaledin} and similarly for the coformality result for configuration spaces.
\end{example}

\bibliographystyle{alpha}
\bibliography{bib}

\end{document}

%% file: macro.tex
\definecolor{red}{RGB}{230,97,0}
\definecolor{blue}{RGB}{93,58,155}


\newcommand{\chdiscr}{\mathsf{discr.\ Ch}}
\newcommand{\sSe}{\mathsf{sSet}}
\newcommand{\sLialg}{\ensuremath{\mathcal{sL}_\infty\text{-}\,\mathsf{alg}}}
\newcommand{\isLialg}{\ensuremath{\infty\text{-}\,\mathcal{sL}_\infty\text{-}\,\mathsf{alg}}}
\newcommand{\sLiealg}{\ensuremath{\mathcal{s}\lie\,\text{-}\,\mathsf{alg}}}
\def\cD{\Delta}

\newcommand{\ucomnalg}{\ensuremath{\mathrm{uCom}_{\leslant 0}\text{-}\,\mathsf{alg}}}
\newcommand{\fnQsp}{\ensuremath{\mathsf{ft}\mathbb{Q}\text{-}\mathsf{ho}(\mathsf{sSet})}}
\newcommand{\fnsp}{\ensuremath{\mathsf{fn}\text{-}\mathsf{Sp}}}
\newcommand{\nQsp}{\ensuremath{\mathsf{n}\mathbb{Q}\text{-}\mathsf{Sp}}}
\newcommand{\nsp}{\ensuremath{\mathsf{n}\text{-}\mathsf{Sp}}}
\newcommand{\fQalg}{\ensuremath{\mathsf{ft}\text{-}\mathsf{ho}(\mathrm{uCom}_{\leqslant 0}\text{-}\,\mathsf{alg})}}

\newcommand{\coker}{\operatorname{coker}}
\newcommand{\Ran}{\operatorname{Ran}}
\newcommand{\ho}{\operatorname{ho}}
\newcommand{\Mal}{\operatorname{Mal}}
\newcommand{\Loc}{\operatorname{Loc}}

\definecolor{red}{RGB}{230,97,0}
\definecolor{blue}{RGB}{93,58,155}
\definecolor{Chocolat}{rgb}{0.36, 0.2, 0.09}
\definecolor{BleuTresFonce}{rgb}{0.215, 0.215, 0.36}
\definecolor{BleuMinuit}{RGB}{0, 51, 102}
\definecolor{bordeau}{rgb}{0.5,0,0}
\definecolor{turquoise}{RGB}{6, 62, 62}

\newcommand{\coline}[1]{\textcolor{bordeau}{#1}}


\newcommand{\sLi}{\ensuremath{\mathcal{sL}_\infty}}
\newcommand{\ccoLi}{\ensuremath{\mathrm{cBCom}}}
\newcommand{\sLie}{\ensuremath{\mathcal{s}\mathrm{Lie}}}
\newcommand{\com}{\ensuremath{\mathrm{Com}}}
\newcommand{\ucom}{\ensuremath{\mathrm{uCom}}}
\newcommand{\Cobar}{\ensuremath{\Omega}}
\newcommand{\hatCobar}{\ensuremath{\widehat{\Omega}}}
\renewcommand{\Bar}{\ensuremath{\mathrm{B}}}
\newcommand{\hatBar}{\ensuremath{\widehat{\mathrm{B}}}}
\newcommand{\lie}{\ensuremath{\mathrm{Lie}}}
\newcommand{\uhocom}{\ensuremath{\mathrm{u}\Omega\mathrm{BCom}}}

\newcommand{\Sh}{\ensuremath{\mathrm{Sh}}}


\newcommand{\g}{\ensuremath{\mathfrak{g}}}
\newcommand{\wsLi}{\ensuremath{\widehat{\mathcal{sL}_\infty}}}


\newcommand{\R}{\ensuremath{\mathrm{R}}}
\newcommand{\Rh}{\ensuremath{\mathrm{R^h}}}
\renewcommand{\L}{\ensuremath{\mathscr{L}}}
\newcommand{\Li}{\mathfrak{L}}
\newcommand{\APL}{\mathrm{A_{PL}}}
\newcommand{\CPL}{\mathrm{C_{PL}}}


\def\QQ{\mathbb{Q}}
\def\Co{\mathbb{C}}
\newcommand{\antishriek}{\text{\raisebox{\depth}{\textexclamdown}}}
\newcommand{\RT}{\mathrm{RT}}
\newcommand{\LRT}{\mathrm{LRT}}
\newcommand{\Sy}{\mathbb{S}}
\renewcommand{\d}{\ensuremath{\mathrm{d}}}
\newcommand{\PP}{\ensuremath{\mathrm{P}}}
\def\Ho#1#2{\Lambda^{#2}_{#1}}
\def\De#1{\Delta^{#1}}
\newcommand{\NN}{\mathbb{N}}
\newcommand{\RR}{\mathbb{R}}
\def\BCH{\mathrm{BCH}}
\newcommand{\wPT}{\ensuremath{\mathrm{wPT}}}
\newcommand{\oPT}{\ensuremath{\overline{\mathrm{PT}}}}
\newcommand{\PaRT}{\ensuremath{\mathrm{PaRT}}}
\newcommand{\PaPT}{\ensuremath{\mathrm{PaPT}}}
\newcommand{\PaPRT}{\ensuremath{\mathrm{PaPRT}}}
\def\rmC{\mathrm{C}}
\newcommand{\berglund}{\ensuremath{\mathcal{B}}}
\def\hot{\widehat{\otimes}} 
\def\whk{\widehat{k}}

\def\colim{\mathop{\mathrm{colim}}}

\newcommand{\CC}{\ensuremath{\mathrm{CC}_\infty}}


\newcommand{\Lalg}{\ensuremath{\mathscr{L}_\infty\text{-}\mathsf{alg}}}

\newcommand{\F}{\ensuremath{\mathrm{F}}}
\newcommand{\h}{\ensuremath{\mathfrak{h}}}

\newcommand{\C}{\ensuremath{\mathscr{C}}}
\newcommand{\D}{\ensuremath{\mathscr{D}}}
\renewcommand{\P}{\ensuremath{\mathscr{P}}}

\newcommand{\VdL}{\ensuremath{\mathrm{VdL}}}
\newcommand{\N}{\ensuremath{\mathrm{N}}}
\newcommand{\ch}{\ensuremath{\mathrm{Ch}}}
\newcommand{\End}{\ensuremath{\mathrm{End}}}
\newcommand{\Aut}{\ensuremath{\mathrm{Aut}}}
\newcommand{\eend}{\ensuremath{\mathrm{end}}}
\newcommand{\susp}{\ensuremath{\mathscr{S}}}
\newcommand{\T}{\ensuremath{\mathcal{T}}}
\newcommand{\vdl}{\ensuremath{\mathrm{VdL}}}
\newcommand{\Tw}{\ensuremath{\mathrm{Tw}}}
\newcommand{\Hom}{\ensuremath{\mathrm{Hom}}}
\renewcommand{\S}{\ensuremath{\mathbb{S}}}
\renewcommand{\k}{\ensuremath{\mathbb{K}}}
\newcommand{\id}{\ensuremath{\mathrm{id}}}
\newcommand{\MC}{\ensuremath{\mathrm{MC}}}
\newcommand{\mc}{\ensuremath{\mathfrak{mc}}}
\newcommand{\mclie}{\ensuremath{\overline{\mathfrak{mc}}}}
\newcommand{\dgl}{\ensuremath{\mathsf{dgLie}}}

\newcommand{\B}{\mathcal{B}}
\newcommand{\Z}{\mathcal{Z}}

\newcommand{\ad}{\operatorname{ad}}
\newcommand{\PTt}{\ensuremath{\widetilde{\mathrm{PT}}}}
\newcommand{\PT}{\ensuremath{\mathrm{PT}}}

\newcommand{\A}{\ensuremath{\mathrm{A}}}

\makeatletter

%% file: Weight.bbl
\def\cprime{$'$}
\begin{thebibliography}{GSNPR05}

\bibitem[Bad02]{Bad02}
B.~Badzioch.
\newblock Algebraic theories in homotopy theory.
\newblock {\em Annals of Mathematics}, 155(3):895--913, 2002.
\newblock {$\mathtt{DOI}$} :
  \href{https://doi.org/10.2307/3062135}{\nolinkurl{10.2307/3062135}}.
  {$\mathtt{arXiv}$} :
  \href{https://arxiv.org/abs/math/0110101}{\nolinkurl{math/0110101}} \medskip.

\bibitem[BdBH21]{BdBH21}
P.~Boavida~de Brito and G.~Horel.
\newblock On the formality of the little disks operad in positive
  characteristic.
\newblock {\em Journal of the London Mathematical Society}, 104(2):634--667,
  2021.
\newblock {$\mathtt{DOI}$} :
  \href{https://doi.org/10.1112/jlms.12442}{\nolinkurl{10.1112/jlms.12442}}.
  {$\mathtt{arXiv}$} :
  \href{https://arxiv.org/abs/1903.09191}{\nolinkurl{1903.09191}} \medskip.

\bibitem[Ber06]{Ber06}
J.~E. Bergner.
\newblock Rigidification of algebras over multi-sorted theories.
\newblock {\em Algebraic \& Geometric Topology}, 6(4):1925--1955, 2006.
\newblock {$\mathtt{DOI}$} :
  \href{https://doi.org/10.2140/agt.2006.6.1925}{\nolinkurl{10.2140/agt.2006.6.1925}}.
  {$\mathtt{arXiv}$} :
  \href{https://arxiv.org/abs/math/0508152}{\nolinkurl{math/0508152}} \medskip.

\bibitem[Ber14]{berglund}
A.~Berglund.
\newblock Homological perturbation theory for algebras over operads.
\newblock {\em Algebraic \& {G}eometric {T}opology}, 14(5):2511--2548, 2014.
\newblock {$\mathtt{DOI}$} :
  \href{https://doi.org/10.2140/agt.2014.14.2511}{\nolinkurl{10.2140/agt.2014.14.2511}}.
  {$\mathtt{arXiv}$} :
  \href{https://arxiv.org/abs/0909.3485v2}{\nolinkurl{0909.3485v2}} \medskip.

\bibitem[Bor74]{borelstable}
A.~Borel.
\newblock Stable real cohomology of arithmetic groups.
\newblock In {\em Annales scientifiques de l'{\'E}cole Normale Sup{\'e}rieure},
  volume~7, pages 235--272, 1974.
\newblock {$\mathtt{DOI}$} :
  \href{https://doi.org/10.24033/asens.1269}{\nolinkurl{10.24033/asens.1269}}
  \medskip.

\bibitem[Bre93]{Bredon97}
G.~E. Bredon.
\newblock {\em Topology and {G}eometry}.
\newblock Number~14 in Graduate Texts in Mathematics. Springer, 1993.
\newblock {$\mathtt{DOI}$} :
  \href{https://doi.org/10.1007/978-1-4757-6848-0}{\nolinkurl{10.1007/978-1-4757-6848-0}}
  \medskip.

\bibitem[BS24]{BS24}
A.~Berglund and R.~Stoll.
\newblock Higher structures in rational homotopy theory.
\newblock {\em Higher Structure and Operadic Calculus, Advanced Courses in
  Mathematics - CRM Barcelona, Springer}, 2024.
\newblock {$\mathtt{arXiv}$} :
  \href{https://arxiv.org/abs/2310.11824}{\nolinkurl{2310.11824}} \medskip.

\bibitem[Caz12]{Cazanave12}
C.~Cazanave.
\newblock {Algebraic homotopy classes of rational functions}.
\newblock {\em Ann. Sci. {\'E}c. Norm. Sup{\'e}r.}, 45(4):511--534, 2012.
\newblock {$\mathtt{DOI}$} :
  \href{https://doi.org/10.24033/asens.2172}{\nolinkurl{10.24033/asens.2172}}.
  {$\mathtt{arXiv}$} :
  \href{https://arxiv.org/abs/0912.2227v2}{\nolinkurl{0912.2227v2}} \medskip.

\bibitem[CC17]{CC17}
D.~{Chataur} and J.~{Cirici}.
\newblock {Rational homotopy of complex projective varieties with normal
  isolated singularities}.
\newblock {\em Forum Mathematicum}, 29(1):41--57, 2017.
\newblock {$\mathtt{DOI}$} :
  \href{https://doi.org/10.1515/forum-2015-0101}{\nolinkurl{10.1515/forum-2015-0101}}.
  {$\mathtt{arXiv}$} :
  \href{https://arxiv.org/abs/1503.05347}{\nolinkurl{1503.05347}} \medskip.

\bibitem[CC22]{CC22}
D.~Chataur and J.~Cirici.
\newblock Sheaves of {{\(E\)}}-infinity algebras and applications to algebraic
  varieties and singular spaces.
\newblock {\em Trans. Am. Math. Soc.}, 375(2):925--960, 2022.
\newblock {$\mathtt{DOI}$} :
  \href{https://doi.org/10.1090/tran/8569}{\nolinkurl{10.1090/tran/8569}}.
  {$\mathtt{arXiv}$} :
  \href{https://arxiv.org/abs/1811.08642}{\nolinkurl{1811.08642}} \medskip.

\bibitem[CH20a]{chuhaugseng}
H.~Chu and R.~Haugseng.
\newblock Enriched {{\(\infty \)}}-operads.
\newblock {\em Adv. Math.}, 361:85, 2020.
\newblock {$\mathtt{DOI}$} :
  \href{https://doi.org/10.1016/j.aim.2019.106913}{\nolinkurl{10.1016/j.aim.2019.106913}}.
  {$\mathtt{arXiv}$} : \href{https://arxiv.org/abs/1707.08049
  }{\nolinkurl{1707.08049 }} \medskip.

\bibitem[CH20b]{CH20}
J.~{Cirici} and G.~{Horel}.
\newblock {Mixed Hodge structures and formality of symmetric monoidal
  functors}.
\newblock {\em Annales {S}cientifiques de l'{É}cole {N}ormale {S}upérieure},
  53(4):1071--1104, 2020.
\newblock {$\mathtt{DOI}$} :
  \href{https://doi.org/10.24033/asens.2440}{\nolinkurl{10.24033/asens.2440}}.
  {$\mathtt{arXiv}$} :
  \href{https://arxiv.org/abs/1703.06816}{\nolinkurl{1703.06816}} \medskip.

\bibitem[CH22]{CH22}
J.~Cirici and G.~Horel.
\newblock Étale cohomology, purity and formality with torsion coefficients.
\newblock {\em Journal of Topology}, 15(4):2270--2297, 2022.
\newblock {$\mathtt{DOI}$} :
  \href{https://doi.org/10.1112/topo.12273}{\nolinkurl{10.1112/topo.12273}}.
  {$\mathtt{arXiv}$} :
  \href{https://arxiv.org/abs/1806.03006}{\nolinkurl{1806.03006}} \medskip.

\bibitem[CH24]{CH22corrigendum}
J.~Cirici and G.~Horel.
\newblock Corrigendum: Étale cohomology, purity and formality with torsion
  coefficients.
\newblock {\em Journal of Topology}, 17(2), 2024.
\newblock {$\mathtt{DOI}$} :
  \href{https://doi.org/10.1112/topo.12348}{\nolinkurl{10.1112/topo.12348}}\medskip.

\bibitem[CRiL24]{Damien}
D.~Calaque and V.~Roca~i {L}ucio.
\newblock Associators from an operadic point of view.
\newblock {\em Higher Structure and Operadic Calculus, Advanced Courses in
  Mathematics - CRM Barcelona, Springer}, 2024.
\newblock {$\mathtt{arXiv}$} :
  \href{https://arxiv.org/abs/2402.05539}{\nolinkurl{2402.05539}} \medskip.

\bibitem[DCH21]{DCH21}
G.~C. Drummond-Cole and G.~Horel.
\newblock Homotopy transfer and formality.
\newblock {\em Annales de l'Institut Fourier}, 71(5):2079--2116, 2021.
\newblock {$\mathtt{DOI}$} :
  \href{https://doi.org/10.5802/aif.3444}{\nolinkurl{10.5802/aif.3444}}.
  {$\mathtt{arXiv}$} :
  \href{https://arxiv.org/abs/1906.03475v1}{\nolinkurl{1906.03475v1}} \medskip.

\bibitem[Del70]{Deligne1}
P.~Deligne.
\newblock Th{\'e}orie de {Hodge}. {I}.
\newblock {\em Actes {Congr}. {Internat}. {Math}.}, 1:425--430, 1970.
\newblock \medskip.

\bibitem[Del71]{Deligne71}
P.~Deligne.
\newblock Th\'eorie de {H}odge. {II}.
\newblock {\em Publications Math\'ematiques de l'IH\'ES}, 40:5--58, 1971.
\newblock {$\mathtt{DOI}$} :
  \href{https://doi.org/10.1007/BF02684692}{\nolinkurl{10.1007/BF02684692}}
  \medskip.

\bibitem[Del74]{Deligne74}
P.~Deligne.
\newblock Th\'eorie de {H}odge. {III}.
\newblock {\em Publications Math\'ematiques de l'IH\'ES}, 44:5--77, 1974.
\newblock {$\mathtt{DOI}$} :
  \href{https://doi.org/10.1007/BF02685881}{\nolinkurl{10.1007/BF02685881}}
  \medskip.

\bibitem[DGM00]{delignealgebre}
P.~Deligne, M.~Goresky, and R.~MacPherson.
\newblock L'alg{\`e}bre de cohomologie du compl{\'e}ment, dans un espace
  affine, d'une famille finie de sous-espaces affines.
\newblock {\em Michigan Mathematical Journal}, 48(1):121--136, 2000.
\newblock {$\mathtt{DOI}$} :
  \href{https://doi.org/10.1307/MMJ/1030132711}{\nolinkurl{10.1307/MMJ/1030132711}}
  \medskip.

\bibitem[DGMS75]{DGMS75}
P.~Deligne, P.~Griffiths, J.~Morgan, and D.~Sullivan.
\newblock Real homotopy theory of {K}\"ahler manifolds.
\newblock {\em Invent. Math.}, 29(3):245--274, 1975.
\newblock {$\mathtt{DOI}$} :
  \href{https://doi.org/10.1007/BF01389853}{\nolinkurl{10.1007/BF01389853}}
  \medskip.

\bibitem[DH18]{DH18}
C.~Dupont and G.~Horel.
\newblock On two chain models for the gravity operad.
\newblock {\em Proceedings of the American Mathematical Society},
  146(5):1845--1910, 2018.
\newblock {$\mathtt{DOI}$} : \href{
  https://doi.org/10.1090/proc/13874}{\nolinkurl{10.1090/proc/13874}}.
  {$\mathtt{arXiv}$} :
  \href{https://arxiv.org/abs/1702.02479}{\nolinkurl{1702.02479}} \medskip.

\bibitem[DM69]{Deligne1969}
P.~Deligne and D.~Mumford.
\newblock The irreducibility of the space of curves of given genus.
\newblock {\em Inst. Hautes \'Etudes Sci. Publ. Math.}, 36:75--109, 1969.
\newblock {$\mathtt{DOI}$} :
  \href{https://doi.org/10.1007/BF02684599}{\nolinkurl{10.1007/BF02684599}}
  \medskip.

\bibitem[Dre15]{Drew15}
B.~Drew.
\newblock Rectification of {D}eligne{'}s mixed {H}odge structures.
\newblock {\em arXiv preprint}, 2015.
\newblock {$\mathtt{arXiv}$} :
  \href{https://arxiv.org/abs/1511.08288v1}{\nolinkurl{1511.08288v1}} \medskip.

\bibitem[Dri90]{Drinfeld90}
V.~G. Drinfeld.
\newblock On quasitriangular quasi-{H}opf algebras and on a group that is
  closely connected with {G}al{$(\overline{\mathbb{Q}}/ \mathbb{Q})$}.
\newblock {\em Algebra i Analiz}, 2:149--181, 1990.
\newblock \medskip.

\bibitem[DSV15]{DotsenkoShadrinVallette15}
V.~Dotsenko, S.~Shadrin, and B.~Vallette.
\newblock De {R}ham cohomology and homotopy {F}robenius manifolds.
\newblock {\em J. Eur. Math. Soc. (JEMS)}, 17(3):535--547, 2015.
\newblock {$\mathtt{DOI}$} : \href{
  https://doi.org/10.4171/JEMS/510}{\nolinkurl{10.4171/JEMS/510}}.
  {$\mathtt{arXiv}$} :
  \href{https://arxiv.org/abs/1203.5077}{\nolinkurl{1203.5077}} \medskip.

\bibitem[Dup16]{Dup16}
C.~Dupont.
\newblock Purity, formality, and arrangement complements.
\newblock {\em International Mathematics Research Notices},
  2016(13):4132--4144, 2016.
\newblock {$\mathtt{DOI}$} : \href{
  https://doi.org/10.1093/imrn/rnv260}{\nolinkurl{10.1093/imrn/rnv260}}.
  {$\mathtt{arXiv}$} :
  \href{https://arxiv.org/abs/1505.00717}{\nolinkurl{1505.00717}} \medskip.

\bibitem[EGNO15]{etingoftensor}
P.~Etingof, S.~Gelaki, D.~Nikshych, and V.~Ostrik.
\newblock {\em Tensor categories}, volume 205 of {\em Math. Surv. Monogr.}
\newblock Providence, RI: AMS, 2015.
\newblock {$\mathtt{DOI}$} : \href{
  https://doi.org/10.1090/surv/205}{\nolinkurl{10.1090/surv/205}} \medskip.

\bibitem[EML53]{EilenbergMacLane53}
S.~Eilenberg and S.~Mac~Lane.
\newblock On the groups of {$H(\Pi,n)$}. {I}.
\newblock {\em Ann. of Math. (2)}, 58:55--106, 1953.
\newblock {$\mathtt{DOI}$} : \href{
  https://doi.org/10.2307/1969820}{\nolinkurl{10.2307/1969820}}\medskip.

\bibitem[Emp24]{Kaledin}
C.~Emprin.
\newblock Kaledin classes and formality criteria.
\newblock {\em arXiv preprint}, 2024.
\newblock {$\mathtt{arXiv}$} :
  \href{https://arxiv.org/abs/2404.17529}{\nolinkurl{2404.17529}} \medskip.

\bibitem[GM13]{GM}
P.~Griffiths and J.~Morgan.
\newblock {\em Rational homotopy theory and differential forms}, volume~16 of
  {\em Prog. Math.}
\newblock New York, NY: Birkh{\"a}user/Springer, 2nd revised and corrected ed.
  edition, 2013.
\newblock {$\mathtt{DOI}$} :
  \href{https://doi.org/10.1007/978-1-4614-8468-4}{\nolinkurl{10.1007/978-1-4614-8468-4}}
  \medskip.

\bibitem[Gro66]{Grothendieck}
A.~Grothendieck.
\newblock On the {De} {Rham} cohomology of algebraic varieties.
\newblock {\em Publ. Math., Inst. Hautes {\'E}tud. Sci.}, 29:95--103, 1966.
\newblock {$\mathtt{DOI}$} :
  \href{https://doi.org/10.1007/BF02684807}{\nolinkurl{10.1007/BF02684807}}
  \medskip.

\bibitem[GSNPR05]{GNPR05}
F.~Guillén~Santos, V.~Navarro, P.~Pascual, and A.~Roig.
\newblock Moduli spaces and formal operads.
\newblock {\em Duke Mathematical Journal}, 129(2):291--335, 2005.
\newblock {$\mathtt{DOI}$} : \href{
  https://doi.org/10.1215/S0012-7094-05-12924-6}{\nolinkurl{10.1215/S0012-7094-05-12924-6}}.
  {$\mathtt{arXiv}$} :
  \href{https://arxiv.org/abs/math/0402098}{\nolinkurl{math/0402098}} \medskip.

\bibitem[Hin97]{Hinich97}
V.~Hinich.
\newblock Homological algebra of homotopy algebras.
\newblock {\em Comm. Algebra}, 25(10):3291--3323, 1997.
\newblock {$\mathtt{DOI}$} : \href{
  https://doi.org/10.1080/00927879708826055}{\nolinkurl{10.1080/00927879708826055}}.
  {$\mathtt{arXiv}$} :
  \href{https://arxiv.org/abs/q-alg/9702015}{\nolinkurl{q-alg/9702015}}
  \medskip.

\bibitem[Hin03]{hinich03bis}
V.~Hinich.
\newblock Tamarkin's proof of {K}ontsevich formality theorem.
\newblock {\em Forum Mathematicum}, 15(4):591--614, 2003.
\newblock {$\mathtt{DOI}$} : \href{
  https://doi.org/10.1515/form.2003.032}{\nolinkurl{10.1515/form.2003.032}}.
  {$\mathtt{arXiv}$} :
  \href{https://arxiv.org/abs/math/0003052}{\nolinkurl{math/0003052}} \medskip.

\bibitem[Hin15]{Hinich15}
V.~Hinich.
\newblock Rectification of algebras and modules.
\newblock {\em {Doc. Math.}}, 20:879--926, 2015.
\newblock {$\mathtt{DOI}$} : \href{
  https://doi.org/10.4171/dm/508}{\nolinkurl{10.4171/dm/508}}.
  {$\mathtt{arXiv}$} :
  \href{https://arxiv.org/abs/1311.4130}{\nolinkurl{1311.4130}} \medskip.

\bibitem[Hin16]{hinichdwyer}
V.~Hinich.
\newblock Dwyer-{K}an localization revisited.
\newblock {\em Homology, Homotopy and Applications}, 18(1):27--48, 2016.
\newblock {$\mathtt{DOI}$} : \href{
  https://doi.org/10.4310/HHA.2016.v18.n1.a3}{\nolinkurl{10.4310/HHA.2016.v18.n1.a3}}.
  {$\mathtt{arXiv}$} :
  \href{https://arxiv.org/abs/1311.4128}{\nolinkurl{1311.4128}} \medskip.

\bibitem[HS79]{HJ79}
S.~Halperin and J.~Stasheff.
\newblock Obstructions to homotopy equivalences.
\newblock {\em Advances in mathematics}, 32(3):233--279, 1979.
\newblock {$\mathtt{DOI}$} :
  \href{https://doi.org/10.1016/0001-8708(79)90043-4}{\nolinkurl{10.1016/0001-8708(79)90043-4}}
  \medskip.

\bibitem[Kal07]{Kal07}
D.~Kaledin.
\newblock Some remarks on formality in families.
\newblock {\em Mosc. Math.~J.}, 7:643--652, 2007.
\newblock {$\mathtt{DOI}$} : \href{
  https://doi.org/10.17323/1609-4514-2007-7-4-643-652}{\nolinkurl{10.17323/1609-4514-2007-7-4-643-652}}.
  {$\mathtt{arXiv}$} :
  \href{https://arxiv.org/abs/math/0509699}{\nolinkurl{math/0509699}} \medskip.

\bibitem[Kee92]{Keel92}
S.~Keel.
\newblock Intersection theory of moduli space of stable {$n$}-pointed curves of
  genus zero.
\newblock {\em Trans. Amer. Math. Soc.}, 330(2):545--574, 1992.
\newblock {$\mathtt{DOI}$} : \href{
  https://doi.org/10.1090/S0002-9947-1992-1034665-0}{\nolinkurl{10.1090/S0002-9947-1992-1034665-0}}
  \medskip.

\bibitem[Kim94]{Kim94}
M.~Kim.
\newblock Weights in cohomology groups arising from hyperplane arrangements.
\newblock {\em Proceedings of the {A}merican {M}athematical {S}ociety.},
  120(3):697--703, 1994.
\newblock {$\mathtt{DOI}$} : \href{
  https://doi.org/10.1090/S0002-9939-1994-1179589-0}{\nolinkurl{10.1090/S0002-9939-1994-1179589-0}}
  \medskip.

\bibitem[Kon99]{kon99}
M.~Kontsevich.
\newblock Operads and motives in deformation quantization.
\newblock {\em Letters in Mathematical Physics}, 48(1):35--72, 1999.
\newblock {$\mathtt{DOI}$} : \href{
  https://doi.org/10.1023/A:1007555725247}{\nolinkurl{10.1023/A:1007555725247}}.
  {$\mathtt{arXiv}$} :
  \href{https://arxiv.org/abs/math/9904055}{\nolinkurl{math/9904055}} \medskip.

\bibitem[Lur17]{Lurie12}
J.~Lurie.
\newblock {\em Higher {A}lgebra}.
\newblock 2017.
\newblock Can be found at
  \href{https://www.math.ias.edu/~lurie/papers/HA.pdf}{Higher Algebra}\medskip.

\bibitem[LV12]{LodayVallette12}
J.-L. Loday and B.~Vallette.
\newblock {\em Algebraic operads}, volume 346 of {\em Grundlehren der
  Mathematischen Wissenschaften}.
\newblock Springer, 2012.
\newblock {$\mathtt{DOI}$} :
  \href{https://doi.org/10.1007/978-3-642-30362-3}{\nolinkurl{10.1007/978-3-642-30362-3}}
  \medskip.

\bibitem[LV14]{lambrechtsformality}
P.~Lambrechts and I.~Voli{\'c}.
\newblock {\em Formality of the little {{\(N\)}}-disks operad}, volume 1079 of
  {\em Mem. Am. Math. Soc.}
\newblock Providence, RI: American Mathematical Society (AMS), 2014.
\newblock {$\mathtt{DOI}$} : \href{
  https://doi.org/10.1090/memo/1079}{\nolinkurl{10.1090/memo/1079}}.
  {$\mathtt{arXiv}$} :
  \href{https://arxiv.org/abs/0808.0457}{\nolinkurl{abs/0808.0457}} \medskip.

\bibitem[Mat06]{mateimassey}
D.~Matei.
\newblock Massey products of complex hypersurface complements.
\newblock In {\em Singularity theory and its applications}, pages 205--219.
  Mathematical Society of Japan, 2006.
\newblock {$\mathtt{DOI}$} :
  \href{https://doi.org/10.2969/ASPM/04310205}{\nolinkurl{10.2969/ASPM/04310205}}.
  {$\mathtt{arXiv}$} :
  \href{https://arxiv.org/abs/math/0505391}{\nolinkurl{math/0505391}} \medskip.

\bibitem[Mor78]{morganalgebraic}
J.~Morgan.
\newblock The algebraic topology of smooth algebraic varieties.
\newblock {\em Publications Math{\'e}matiques de l'IH{\'E}S}, 48:137--204,
  1978.
\newblock {$\mathtt{DOI}$} :
  \href{https://doi.org/10.1007/BF02684316}{\nolinkurl{10.1007/BF02684316}}
  \medskip.

\bibitem[MR19]{MR19}
V.~Melani and M.~Rubió.
\newblock Formality criteria for algebras over operads.
\newblock {\em Journal of Algebra}, 529:65--88, 2019.
\newblock {$\mathtt{DOI}$} :
  \href{https://doi.org/10.1016/j.jalgebra.2019.03.016}{\nolinkurl{10.1016/j.jalgebra.2019.03.016}}.
  {$\mathtt{arXiv}$} :
  \href{https://arxiv.org/abs/1712.09229v2}{\nolinkurl{1712.09229v2}} \medskip.

\bibitem[NA87]{NavarroAznar87}
V.~Navarro-Aznar.
\newblock Sur la th\'eorie de {H}odge--{D}eligne.
\newblock {\em Invent. {M}ath.}, 90:11--76, 1987.
\newblock {$\mathtt{DOI}$} :
  \href{https://doi.org/10.1007/BF01389031}{\nolinkurl{10.1007/BF01389031}}
  \medskip.

\bibitem[Pet14]{Petersen14}
D.~Petersen.
\newblock Minimal models, {GT}-action and formality of the little disk operad.
\newblock {\em Selecta Math.}, 20(3):817--822, 2014.
\newblock {$\mathtt{DOI}$} :
  \href{https://doi.org/10.1007/s00029-013-0135-5}{\nolinkurl{10.1007/s00029-013-0135-5}}.
  {$\mathtt{arXiv}$} :
  \href{https://arxiv.org/abs/1303.1448}{\nolinkurl{1303.1448}} \medskip.

\bibitem[Pet22]{petersenremark}
D.~Petersen.
\newblock A remark on singular cohomology and sheaf cohomology.
\newblock {\em Math. Scand.}, 128(2):229--238, 2022.
\newblock {$\mathtt{DOI}$} :
  \href{https://doi.org/10.7146/math.scand.a-132191}{\nolinkurl{10.7146/math.scand.a-132191}}.
  {$\mathtt{arXiv}$} :
  \href{https://arxiv.org/abs/2102.06927}{\nolinkurl{2102.06927}} \medskip.

\bibitem[PY99]{papadimarational}
S.~Papadima and S.~Yuzvinsky.
\newblock On rational {$K [\pi, 1]$} spaces and {K}oszul algebras.
\newblock {\em Journal of Pure and Applied Algebra}, 144(2):157--167, 1999.
\newblock {$\mathtt{DOI}$} :
  \href{https://doi.org/10.1016/S0022-4049(98)00058-9}{\nolinkurl{10.1016/S0022-4049(98)00058-9}}
  \medskip.

\bibitem[RiL22]{RL2}
V.~Roca~i Lucio.
\newblock Curved operadic calculus.
\newblock {\em Bulletin de la Société Mathématique de France},
  152(1):45--147, 2022.
\newblock {$\mathtt{arXiv}$} :
  \href{https://arxiv.org/abs/2201.07155}{\nolinkurl{2201.07155}} \medskip.

\bibitem[Sal17]{Sal17}
B.~Saleh.
\newblock Noncommutative formality implies commutative and {L}ie formality.
\newblock {\em Algebraic \& Geometric Topology}, 17(4):2523--2542, 2017.
\newblock {$\mathtt{DOI}$} :
  \href{https://doi.org/10.2140/agt.2017.17.2523}{\nolinkurl{10.2140/agt.2017.17.2523}}.
  {$\mathtt{arXiv}$} :
  \href{https://arxiv.org/abs/1609.02540v2}{\nolinkurl{1609.02540v2}}\medskip.

\bibitem[SGA]{SGA4}
SGA{4}.
\newblock Théorie des topos et cohomologie \'etale des sch\'emas. {T}ome 3.
  {S}éminaire de {G}éométrie {A}lgébrique du {B}ois-{M}arie, dirigé par
  {M}. {A}rtin, {A}. {G}rothendieck, {J}.-{L}. {V}erdier, 1963-1964.
\newblock {\em LSLN Springer-Verlag}, (269, 270, 305).
\newblock {$\mathtt{DOI}$} :
  \href{https://doi.org/10.1007/BFb0070714}{\nolinkurl{10.1007/BFb0070714}}\medskip.

\bibitem[Shi23]{shin}
T.~Shin.
\newblock Prismatic cohomology and {{\(p\)}}-adic homotopy theory.
\newblock {\em J. Homotopy Relat. Struct.}, 18(4):521--541, 2023.
\newblock {$\mathtt{DOI}$} :
  \href{https://doi.org/10.1007/s40062-023-00335-0}{\nolinkurl{10.1007/s40062-023-00335-0}}.
  {$\mathtt{arXiv}$} : \href{https://arxiv.org/abs/2107.02256
  }{\nolinkurl{2107.02256 }}\medskip.

\bibitem[Tam98]{Tam98}
D.~E. Tamarkin.
\newblock Another proof of {M}. {K}ontsevich formality theorem.
\newblock 1998.
\newblock {$\mathtt{DOI}$} : \href{https://doi.org/
  10.48550/arXiv.math/9803025}{\nolinkurl{ 10.48550/arXiv.math/9803025}}.
  {$\mathtt{arXiv}$} :
  \href{https://arxiv.org/abs/math/9803025}{\nolinkurl{math/9803025}}\medskip.

\bibitem[Vai19]{vaintrobmoduli}
D.~Vaintrob.
\newblock Moduli of framed formal curves.
\newblock {\em arXiv preprint
  \href{https://arxiv.org/abs/1910.11550}{\nolinkurl{1910.11550}}}, 2019.
\newblock \medskip.

\bibitem[Vai21]{vaintrobformality}
D.~Vaintrob.
\newblock Formality of little disks and algebraic geometry.
\newblock {\em arXiv e-print
  \href{https://arxiv.org/abs/2103.15054}{\nolinkurl{2103.15054}}}, 2021.
\newblock \medskip.

\bibitem[VdL03]{vdL03}
P.~Van~der {L}aan.
\newblock Coloured {K}oszul duality and strongly homotopy operads.
\newblock {\em arXiv preprint}, 2003.
\newblock {$\mathtt{arXiv}$} :
  \href{https://arxiv.org/abs/math/0312147}{\nolinkurl{math/0312147}}\medskip.

\bibitem[Voi02]{Voi02}
C.~Voisin.
\newblock {\em Hodge Theory and Complex Algebraic Geometry {I}}.
\newblock Cambridge Studies in Advanced Mathematics. Cambridge University
  Press, 2002.
\newblock {$\mathtt{DOI}$} :
  \href{https://doi.org/10.1017/CBO9780511615344}{\nolinkurl{
  10.1017/CBO9780511615344}}\medskip.

\bibitem[War21]{War19}
B.~Ward.
\newblock Massey {P}roducts for {G}raph {H}omology.
\newblock {\em International Mathematics Research Notices},
  2022(11):8086--8161, 01 2021.
\newblock {$\mathtt{DOI}$} : \href{https://doi.org/
  10.1093/imrn/rnaa346}{\nolinkurl{10.1093/imrn/rnaa346}}. {$\mathtt{arXiv}$} :
  \href{https://arxiv.org/abs/1903.12055v3}{\nolinkurl{1903.12055v3}}\bigskip.

\end{thebibliography}
